\documentclass[11pt]{amsart}
\usepackage[latin1]{inputenc}
\usepackage{amsfonts}
\usepackage{amsthm}
\usepackage{amsmath}
\usepackage{amssymb}
\usepackage{amscd}
\usepackage{verbatim}
\usepackage{multicol}
\usepackage{tabularx}

\theoremstyle{definition}
\newtheorem{thm}{Theorem}[section]
\newtheorem{prop}[thm]{Proposition}
\newtheorem{cor}[thm]{Corollary}

\newtheorem{lem}[thm]{Lemma}
\newtheorem{defn}[thm]{Definition}
\newtheorem{ex}[thm]{Example}

\newtheorem*{rem}{Remark}

\numberwithin{equation}{section}

\def\cal#1{\text{$\mathcal{#1}$}}

\def\ord#1^#2{#1$^{\text{#2}}$}
\def\foral{\,\forall\,}
\def\lie#1{\mathfrak{#1}}
\def\tlie#1{\tilde{\mathfrak{#1}}}
\def\hlie#1{\hat{\mathfrak{#1}}}

\def\uqr#1^#2{\text{$U_q^{#2}(\lie #1)$}}

\def\uqhr#1^#2{\text{$U_q^{#2}(\hlie #1)$}}
\def\us#1^#2{\text{$U_{\xi}^{#2}(\lie #1)$}}
\def\ush#1^#2{\text{$U_{\xi}^{#2}(\hlie #1)$}}
\def\dus#1^#2{\text{$\dot{U}_{\xi}^{#2}(\lie #1)$}}
\def\dush#1^#2{\text{$\dot{U}_{\xi}^{#2}(\hlie #1)$}}
\def\gb#1{{\mbox{\boldmath $#1$}}}

\def\wt{{\rm wt}}

\def\ch{{\rm char}}

\def\opl_#1^#2{\text{\scriptsize$\bigoplus\limits_{\text{\footnotesize$#1$}}^{\text{\footnotesize$#2$}}$}}
\def\otm_#1^#2{\text{\scriptsize$\bigotimes\limits_{\text{\footnotesize$#1$}}^{\text{\footnotesize$#2$}}$}}

\renewcommand{\thefootnote}

\begin{document}

\flushbottom
%%%%%%%%%%%%%%%%%%%%%%%% Topmatter %%%%%%%%%%%%%%%%%%%%%%%%%

\title[Finite-Dimensional Representations of Hyper Loop Algebras]{Finite-Dimensional Representations of\\ Hyper Loop Algebras Over\\ Non-Algebraically Closed Fields}

\author[D. Jakeli\'c and A. Moura]{Dijana Jakeli\'c and Adriano Moura}
\address{Max-Planck-Institut f\"ur Mathematik, D-53111, Bonn, Germany and \vspace{-12pt}}
\address{Department of Mathematics, Statistics, and Computer Science, \vspace{-12pt}}
\address{University of Illinois at Chicago, Chicago, IL 60607-7045} \email{dijana@math.uic.edu}

\address{UNICAMP - IMECC, Campinas - SP, 13083-970, Brazil.} \email{aamoura@ime.unicamp.br}

\maketitle
\centerline{\small{
\begin{minipage}{350pt}
{\bf Abstract:} We study finite-dimensional representations of hyper loop algebras over non-algebraically closed fields. The main results concern the classification of the irreducible representations, the construction of the Weyl modules, base change,  tensor products of irreducible and Weyl modules, and the block decomposition of the underlying abelian category. Several results are interestingly related to the study of irreducible representations of polynomial algebras and Galois theory.
\end{minipage}}}


\section*{Introduction}
Let $\lie g$ be a finite-dimensional simple Lie algebra over the
complex numbers with a fixed triangular decomposition $\lie g=\lie
n^-\oplus\lie h\oplus\lie n^+$, let $\tlie g=\lie g\otimes \mathbb
C[t,t^{-1}] = \tlie n^-\oplus\tlie h\oplus\tlie n^+$ be its loop
algebra and $\mathbb F$ an algebraically closed field. In \cite{jm},
the authors initiated the study of finite-dimensional
representations of the hyper loop algebra $U(\tlie g)_\mathbb F$ of
$\lie g$ over $\mathbb F$. This algebra is constructed from
Garland's integral form $U(\tlie g)_\mathbb Z$ of the universal
enveloping algebra $U(\tlie g)$ by tensoring it with $\mathbb F$
over $\mathbb Z$. It contains the hyperalgebra $U(\lie g)_\mathbb F$
of the corresponding universal Chevalley group as a subalgebra. Let
us note that by viewing \cite{jm} in an appropriate algebraic
geometric framework (such as the ones of
\cite{garalg,kumar,mat,tits}), it should turn out that studying
finite-dimensional representations of these algebras is equivalent
to studying finite-dimensional representations of certain algebraic
loop groups (in the defining characteristic).

The passage from algebraically closed fields to non-algebraically closed ones is the subject of the present paper. Let us explain what the main reason is for this subject to be much more interesting in the
context of hyper loop algebras than in the context of the
hyperalgebra $U(\lie g)_\mathbb F$. The subalgebra $U(\lie
h)_\mathbb F$ has a natural set of generators -- the images of the
basis vectors of $U(\tlie g)_\mathbb Z$ belonging to $U(\lie h)$. It
turns out that the eigenvalues of the action of these generators on
a finite-dimensional $U(\lie g)_\mathbb F$-module lie in the prime
field $\mathbb P$ of $\mathbb F$. One can then show that the functor
determined by the assignment $V\mapsto V\otimes_\mathbb P\mathbb F$
induces a bijection from the set of isomorphism classes of
finite-dimensional irreducible $U(\lie g)_\mathbb P$-modules  and
that of $U(\lie g)_\mathbb F$-modules (and also, more generally,
between the corresponding sets of isomorphism classes of
finite-dimensional highest-weight modules). On the other hand, when working with hyper
loop algebras, the image of a finite-dimensional irreducible
$U(\tlie g)_\mathbb P$-module under the map $V\mapsto
V\otimes_\mathbb P\mathbb F$ may even be reducible. The reason behind
this difference in behavior is the following fact which is easily
deduced from \cite{jm}. Given any set $\cal G$ of generators of
$U(\tlie h)_\mathbb F$ and any $a\in\mathbb F$, there exist a
finite-dimensional $U(\tlie g)_\mathbb F$-module $V$ and $\Lambda\in
\cal G$ such that $a$ is an eigenvalue of the action of $\Lambda$ on
$V$.

Let $\mathbb K$ be a field having $\mathbb F$ as an algebraic closure. It turns out that the action of $U(\tlie h)_\mathbb K$ on a finite-dimensional $U(\tlie g)_\mathbb K$-module is completely determined by the action of a certain polynomial subalgebra $\mathbb K[\gb\Lambda]$ of $U(\tlie h)_\mathbb K$. When $\mathbb K=\mathbb F$ this is precisely the reason behind  the identification of the set $\cal P_\mathbb F^+$ of dominant $\ell$-weights  with Drinfeld polynomials and, more generally, the identification  of the $\ell$-weight lattice $\cal P_\mathbb F$ of $U(\tlie g)_\mathbb F$ with the multiplicative group of $I$-tuples of rational functions in one variable with constant term $1$. Here $I$ stands for the set of nodes of the Dynkin diagram of $\lie g$.

In general, the finite-dimensional irreducible $U(\tlie g)_\mathbb K$-modules are not highest-$\ell$-weight modules, but rather  highest-quasi-$\ell$-weight modules. That is, they are generated by an irreducible $\mathbb K[\gb\Lambda]$-module on which the augmentation ideal of $U(\tlie n^+)_\mathbb K$ acts trivially. When this irreducible $\mathbb K[\gb\Lambda]$-module is one-dimensional the corresponding $U(\tlie g)_\mathbb K$-modules are highest-$\ell$-weight in the usual sense. Each irreducible $\mathbb K[\gb\Lambda]$-module is isomorphic to a module of the form $\mathbb K[\gb\Lambda]v$ where $v$ is a nonzero vector of an irreducible $\mathbb F[\gb\Lambda]$-module. Since $\mathbb F$ is algebraically closed, all irreducible  $\mathbb F[\gb\Lambda]$-modules are one-dimensional and their isomorphism classes map bijectively onto the set $\gb M_\mathbb F$ of maximal ideals of $\mathbb F[\gb\Lambda]$. Let $\cal K(\gb\varpi)$ be an irreducible $\mathbb K[\gb\Lambda]$-module corresponding to $\gb\varpi\in\gb M_\mathbb F$, define an equivalence relation on $\gb M_\mathbb F$ by setting $\gb\varpi\equiv\gb\pi$ iff $\cal K(\gb\varpi)$ is isomorphic to $\cal K(\gb\pi)$, and denote by $\gb M_{\mathbb F,\mathbb K}$ the corresponding set of equivalence classes. The set $\cal P_\mathbb F$ can be naturally included in $\gb M_\mathbb F$. Among the irreducible $\mathbb K[\gb\Lambda]$-modules, those occurring as highest-quasi-$\ell$-weight spaces of a finite-dimensional $U(\tlie g)_\mathbb K$-module  correspond to elements of $\cal P_{\mathbb F,\mathbb K}^+$ where $\cal P_{\mathbb F,\mathbb K}^+$ is the subset of $\gb M_{\mathbb F,\mathbb K}$ corresponding to the set $\cal P_\mathbb F^+$.

After establishing the above results, we proceed with the study of finite-dimensional irreducible $U(\tlie g)_\mathbb K$-modules and the corresponding Weyl modules. In particular, we show that the isomorphism classes of irreducible $U(\tlie g)_\mathbb K$-modules (as well as those of Weyl modules) are parametrized by $\cal P_{\mathbb F,\mathbb K}^+$. We remark that the proofs of these results do not depend on the results of \cite{jm}, i.e., on the knowledge of the corresponding results when $\mathbb K=\mathbb F$. Consequently, the corresponding results of \cite{jm} are recovered by setting $\mathbb K=\mathbb F$. We shall denote by $V_\mathbb K(\gb\omega)$ and $W_\mathbb K(\gb\omega)$ the irreducible and Weyl module, respectively,  corresponding to the equivalence class of $\gb\omega\in\cal P_\mathbb F^+$.

Our next topic is the study of base change, i.e., the relations between $V_\mathbb K(\gb\omega)$ and $V_\mathbb
F(\gb\omega)$ and between $W_\mathbb K(\gb\omega)$ and $W_\mathbb F(\gb\omega)$. We introduce the concept of a $U(\tlie g)_\mathbb K$-ample-form of a finite-dimensional $U(\tlie g)_\mathbb F$-module $V$. This is a finite-dimensional $U(\tlie g)_\mathbb K$-submodule of $V$ which spans $V$ over $\mathbb F$. The adjective ``ample'' is chosen to stress that an ample-form may be ``too large'', i.e., its $\mathbb K$-dimension may be larger than the $\mathbb F$-dimension
of $V$. For instance, the $\mathbb K[\gb\Lambda]$-module $\cal
K(\gb\varpi)$ is a  $\mathbb K[\gb\Lambda]$-ample-form of the
corresponding one-dimensional $\mathbb F[\gb\Lambda]$-module, but it
is clearly not a form in general. Set $\deg(\gb\varpi)=\dim_\mathbb
K(\cal K(\gb\varpi))$. Our main result on ample-forms, Theorem
\ref{t:forms}, may be regarded  as an analogue of the
conjecture in \cite{jm} (the conjecture is in the context of
ample-forms over discrete valuation rings). Namely, we prove that
$V_\mathbb K(\gb\omega)$ is an ample-form of $V_\mathbb
F(\gb\omega)$ and $\dim_\mathbb K(V_\mathbb
K(\gb\omega))=\deg(\gb\omega)\dim_\mathbb F(V_\mathbb
F(\gb\omega))$, and similarly for the Weyl modules. We additionally
prove that if $V$ is a finite-dimensional highest-$\ell$-weight
$U(\tlie g)_\mathbb F$-module of highest $\ell$-weight
$\gb\omega\in\cal P_\mathbb F^+$, it admits a minimal  $U(\tlie
g)_\mathbb K$-ample-form $V^{f_\mathbb K}$. Moreover,
$\dim_\mathbb K(V^{f_\mathbb K})\ge \deg(\gb\omega)\dim_\mathbb
F(V)$ and strict inequality may happen even when
$\deg(\gb\omega)=1$. We also use the theory of ample-forms to prove
that the quasi-$\ell$-weights of a finite-dimensional $U(\tlie
g)_\mathbb K$-module lie in $\cal P_{\mathbb F,\mathbb K}$.

It was proved in \cite{cint,cpnew,jm} that if $\gb\varpi,\gb\pi\in\cal P_\mathbb
F^+$ are relatively prime then $V_\mathbb F(\gb\varpi\gb\pi)$ is
isomorphic to $V_\mathbb F(\gb\varpi)\otimes V_\mathbb F(\gb\pi)$.
When
$\mathbb K$ is in place of $\mathbb F$, we show in Theorem \ref{t:tp} that the corresponding result  holds iff $\deg(\gb\varpi\gb\pi)
= \deg(\gb\varpi)\deg(\gb\pi)$. It is easy to see that one always
has $\deg(\gb\varpi\gb\pi) \le \deg(\gb\varpi)\deg(\gb\pi)$  and
that strict inequality does happen. We give two examples  showing
that $V_\mathbb K(\gb\varpi)\otimes V_\mathbb K(\gb\pi)$ may be
reducible in case of strict inequality. It becomes clear from these
examples that, contrary to the case $\mathbb K=\mathbb F$, the
combinatorics related to the corresponding Clebsch-Gordan problem
for finite-dimensional irreducible $U(\tlie g)_\mathbb K$-modules,
when $\mathbb K\ne \mathbb F$, is not immediately reduced to the
Clebsch-Gordan problem for $U(\lie g)_\mathbb K$-modules. In fact,
it is clear that Galois theory plays a crucial role in its
solution alongside Theorem \ref{t:forms}. We will further explore
this subject, as well as other applications of Theorem \ref{t:forms}
to multiplicity problems in the category of finite-dimensional
representations of $U(\tlie g)_\mathbb K$ in the forthcoming
publication \cite{jmtpec}.

In the last section, we study the block decomposition of the
category of finite-dimensional $U(\tlie g)_\mathbb K$-modules. The
overall scheme of the study is the same as that used in
\cite{emec,cmqc}, \cite{cmsc}, and \cite{jm}  to study the blocks of
the categories of finite-dimensional representations of quantum,
classical, and hyper loop algebras, respectively (see also
\cite{cg}). Namely, the study is split in two parts -- the first one
consists in showing that the category in question is a direct sum of certain
abelian subcategories and the second one in showing that these
subcategories are indecomposable. These subcategories consist of
those modules whose quasi-$\ell$-weights determine the same
element $\chi\in\Xi_\mathbb K$ where $\Xi_\mathbb K$ is the set of
equivalence classes of the equivalence relation  on $\cal P_{\mathbb
F,\mathbb K}$ induced by the $\ell$-root lattice $\cal Q_\mathbb F$.
The $\ell$-root lattice is the subgroup of $\cal P_\mathbb F$
corresponding to the root lattice of $\lie g$ and $\Xi_\mathbb F$ is
the quotient group $ \cal P_\mathbb F/\cal Q_\mathbb F$. The
elements of $\Xi_\mathbb K$ are called spectral characters (the
elements of the quantum analogue $\Xi_q$ of $\Xi_\mathbb K$ are
called elliptic characters since they have an elliptic behavior with
respect to the quantum parameter $q$).
Hence, the first step in the study of blocks is equivalent to proving that all indecomposable finite-dimensional $U(\tlie g)_\mathbb K$-modules have a well-defined spectral character. By using our results on ample-forms, this is reduced to proving the claim when $\mathbb K=\mathbb F$ in which case it can be reformulated as follows: the $\ell$-weights of  an indecomposable finite-dimensional $U(\tlie g)_\mathbb F$-module differ by an element of the $\ell$-root lattice $\cal Q_\mathbb F$. In characteristic zero this follows from  \cite{cmsc}. The positive characteristic case follows from the characteristic zero one using a corollary of the conjecture of \cite{jm}.
For the second step, the tensor product results were the main tools used in \cite{emec,cmsc,cmqc,jm} providing a way of constructing  a sequence of indecomposable modules linking two irreducible ones having the same spectral character. The same technique does not work when $\mathbb K$ is not algebraically closed since tensor products are not as well-behaved. However, by combining this tensor product technique over $\mathbb F$ together with our results on ample-forms, we are able to prove that the subcategories corresponding to spectral characters are indeed indecomposable. We remark that the arguments used in this step of the study of blocks are independent from the first step and, therefore, do not depend on the conjecture in \cite{jm}. In particular, it follows  that the irreducible $U(\tlie g)_\mathbb K$-modules whose highest quasi-$\ell$-weights determine the same spectral character must lie in the same  block.

The paper is organized as follows. In Section \ref{ss:irpa} we introduce the notation and review the main results concerning representations of polynomial algebras. For the reader's convenience, given the apparent lack  of an explicit reference, we provide the proofs of the main results of this section in the appendix. The notation concerning simple Lie algebras and their loop algebras is fixed in Section \ref{ss:simpleloop}. The construction of hyperalgebras is reviewed in Section \ref{ss:hyper} and the theory of finite-dimensional representations of $U(\lie g)_\mathbb K$ is briefly reviewed in Section \ref{ss:repg}. In Section \ref{ss:qewm} we introduce a general notion of quasi-$\ell$-weight modules and establish that the quasi-$\ell$-weights occurring as highest-quasi-$\ell$-weights of finite-dimensional representations must lie in $\cal P_{\mathbb F,\mathbb K}^+$. The Weyl modules and their irreducible quotients are studied in Section \ref{ss:wm}. Ample-forms are studied in Section \ref{ss:forms} and tensor products in Section \ref{ss:tp}. Section \ref{ss:blocks} ends the paper with the study of blocks.

\vskip15pt
\noindent{\bf Acknowledgements:} D.J. is pleased to thank the Max-Planck-Institut f\"ur Mathematik in Bonn for its hospitality and support. The research of A.M. is partially supported by CNPq and FAPESP. We
thank S. Ovsienko and L. San-Martin for useful discussions. We also thank the referee for the question regarding tensor products of Weyl modules.

\section{Preliminaries}

\subsection{Basics on Representations of Commutative Algebras and Field Extensions}\label{ss:irpa}
\hfill\\

Throughout the paper, let $\mathbb F$ be a fixed algebraically closed field and $\mathbb K$ a  subfield of  $\mathbb F$ having $\mathbb F$ as an algebraic closure. Also,
$\mathbb C, \mathbb R,\mathbb Z,\mathbb Z_+,\mathbb N$ denote the sets of complex numbers, reals, integers, non-negative integers, and positive integers, respectively. Given a ring $\mathbb A$, the underlying multiplicative group of units is denoted by $\mathbb A^\times$. The dual of a vector space $V$ is denoted by $V^*$. The symbol $\cong$ means ``isomorphic to''.

Let $\mathbb A[X]$ be the polynomial algebra whose variables are the elements of the set $X$ with coefficients in a ring $\mathbb A$. When talking about the minimal polynomial of the action of an element of $\mathbb K[X]$ on a finite-dimensional $\mathbb K[X]$-module we use $u$ as variable of that polynomial to avoid confusion. Similarly, the irreducible polynomial of an element $a\in\mathbb F$ over $\mathbb K$ will be an element of the polynomial ring $\mathbb K[u]$. The ring of formal power series in one variable with coefficients in $\mathbb A$ will be denoted by $\mathbb A[[u]]$.

Given an associative unitary $\mathbb K$-algebra $A$, an $A$-module
is a $\mathbb K$-vector space $\cal V$ equipped with an $A$-action,
i.e., a homomorphism of associative unitary $\mathbb K$-algebras
$A\to {\rm End}_{\mathbb K}(\cal V)$.  The action of $a\in A$ on
$v\in\cal V$ is denoted simply by $av$. The expression ``$\cal V$ is
a finite-dimensional $A$-module'' means that $\cal V$ is an
$A$-module which is finite-dimensional as a $\mathbb K$-vector
space.

If $\cal V$ is a $\mathbb K$-vector space, set $\cal V^{\mathbb F}=\cal V\otimes_{\mathbb K}\mathbb F$ and identify $\cal V$ with $\cal V\otimes 1\subseteq \cal V^\mathbb F$.
If $A$ is a $\mathbb K$-algebra and $\cal V$  is an $A$-module, then $\cal V^\mathbb F$ is naturally an $A^\mathbb F$-module. Moreover, if $\cal V$ is an irreducible $A$-module, $\cal V$ is isomorphic to a quotient of $A$ by a maximal ideal and, if $A$ is commutative and $a\in A$, Schur's Lemma implies that  either $a\cal V=0$ or $a\in{\rm Aut}_A(\cal V)$.

Since $\mathbb F$ is algebraically closed, all irreducible $\mathbb F[X]$-modules are one-dimensional over $\mathbb F$ and determined by the eingenvalues $\varpi_u$ of each $u\in X$ acting on a fixed $\mathbb F$-basis of this module. In terms of maximal ideals, this is equivalent to saying that the maximal ideals of $\mathbb F[X]$ are the ideals generated by the sets of the form $\{x-\varpi_x:x\in X, \varpi_x\in\mathbb F\}$. Let us restate this in terms of multiplicative functionals $\gb\varpi\in \mathbb F[X]^*$. Given a field $\mathbb L$, $\gb\varpi:\mathbb L[X]\to\mathbb L$ is said to be multiplicative if it is a homomorphism of unitary $\mathbb L$-algebras. We denote by $\gb M_\mathbb L(X)$ the subset of $\mathbb L[X]^*$ consisting of the multiplicative functionals.
In particular, if $\gb\varpi\in\gb M_\mathbb L(X)$,  $\gb\varpi$ is completely determined by its values $\varpi_x:=\gb\varpi(x)$ for $x\in X$. We denote by $\cal F(\gb\varpi)$ the one-dimensional $\mathbb F[X]$-module determined by $\gb\varpi\in\gb M_\mathbb F(X)$. Given a nonzero vector $v\in\cal F(\gb\varpi)$ and a subfield $\mathbb L$ of $\mathbb F$ such that $\varpi_x$ is algebraic over $\mathbb L$ for all $x\in X$, let
\begin{equation}
\cal L(\gb\varpi):=\mathbb L[X]v.
\end{equation}
One can easily check that the isomorphism class of $\cal L(\gb\varpi)$ does not depend on the choice of $v$ and that, as an $\mathbb L$-vector space, $\cal L(\gb\varpi)$ is isomorphic to the field $\mathbb L(\gb\varpi)$ obtained from $\mathbb L$ by adjoining the elements $\varpi_x$ (see also Section \ref{ss:aa-c}). We define the degree of $\gb\varpi$ over $\mathbb L$ to be
\begin{equation}
\deg_\mathbb L(\gb\varpi)=\dim_\mathbb L(\cal L(\gb\varpi))=[\mathbb L(\gb\varpi):\mathbb L].
\end{equation}
Set
\begin{equation}
\gb M_\mathbb F^{f_\mathbb L}(X) = \{\gb\varpi\in\gb M_\mathbb F(X): \deg_\mathbb L(\gb\varpi)<\infty\}
\end{equation}
$\gb M_\mathbb L(X)$ can be naturally regarded as a subset of $\gb M_\mathbb F^{f_\mathbb L}(X)$.

\begin{defn}\label{d:cong}
Two elements $\gb\varpi,\gb\varpi'\in\gb M_\mathbb F(X)$ are said to be conjugate over $\mathbb L$, written $\gb\varpi\equiv_\mathbb L\gb\varpi'$, if there exists $g\in{\rm Aut}(\mathbb F/\mathbb L)$ such that $\varpi'_x=g(\varpi_x)$ for all $x\in X$. In that case we write $\gb\varpi'=g(\gb\varpi)$. The conjugacy class of $\gb\varpi$ will be denoted by $[\gb\varpi]_\mathbb L$ and the set of all conjugacy classes will be denoted by $\gb M_{\mathbb F,\mathbb L}(X)$. The subset of $\gb M_{\mathbb F,\mathbb L}(X)$ corresponding to elements of $\gb M_\mathbb F^{f_\mathbb L}(X)$ will be denoted by $\gb P_\mathbb L(X)$.
\end{defn}

Notice that the conjugacy classes of elements from $\gb M_\mathbb L(X)$ are singletons and, hence, we can identify $\gb M_\mathbb L(X)$ with a subset of $\gb P_\mathbb L(X)$.

From now on, conjugation will always be  over our fixed field $\mathbb K$. Thus, we simplify the notation and simply write $[\gb\varpi], \gb\varpi\equiv\gb\varpi'$, and $\deg(\gb\varpi)$ dropping the reference on $\mathbb K$. We also suppose that the set $X$ is fixed and further simplify the notation by writing $\gb M_\mathbb F,\gb M_\mathbb K, \gb M_\mathbb F^{f_\mathbb K}, \gb M_{\mathbb F,\mathbb K}$, and $\gb P_\mathbb K$ in place of $\gb M_\mathbb F(X)$ and so forth.

The following theorem will be strongly used in the remainder of the paper. For the reader's convenience, we include a proof in the Appendix.

%\pagebreak
\begin{thm}\label{t:irpa}Let $\gb\varpi,\gb\varpi'\in\gb M_\mathbb F$ and $\gb\omega\in\gb M_\mathbb F^{f_\mathbb K}$.
\begin{enumerate}
\item $\cal K(\gb\varpi)$ is an irreducible $\mathbb K[X]$-module.
\item $\cal K(\gb\varpi)\cong\cal K(\gb\varpi')$ iff $\gb\varpi\equiv\gb\varpi'$.
\item $\cal K(\gb\varpi)$ is isomorphic to the restriction of $\cal L(\gb\varpi)$ to $\mathbb K[X]$, where $\mathbb L=\mathbb K(\gb\varpi)$.
\item The assignment $\gb\omega\mapsto \cal K(\gb\omega)$ induces a bijection from  $\gb P_\mathbb K$ onto the set of isomorphism classes of finite-dimensional irreducible $\mathbb K[X]$-modules.
\item $\cal K(\gb\omega)^\mathbb F\cong \opl_{\gb\omega'\in[\gb\omega]}^{} V_{\gb\omega'}$ with $V_{\gb\omega'}$ an indecomposable self-extension of  $\cal F(\gb\omega')$ of length $\deg(\gb\omega)/|[\gb\omega]|$, where $|[\gb\omega]|$ is the cardinality of $[\gb\omega]$.
\end{enumerate}
\end{thm}

\begin{cor}\label{c:irpa}
If $V$ is a finite-dimensional $\mathbb K[X]$-module, then $V=\opl_{[\gb\varpi]\in\gb P_\mathbb K}^{} V_{\gb\varpi}$, where $V_\gb\varpi$ is either trivial or a $\mathbb K[X]$-module whose irreducible constituents are isomorphic to $\cal K(\gb\varpi)$.\hfill\qedsymbol
\end{cor}

\begin{rem}
One easily concludes from the proof of Theorem \ref{t:irpa} that, if $\gb\omega\in\gb M_\mathbb F^{f_\mathbb K}$ is such that $\omega_x$ is separable over $\mathbb K$ for all $x\in X$, then $\deg(\gb\omega)=|[\gb\omega]|$. On the other hand, if $\omega_x$ is purely inseparable over $\mathbb K$ for every $x$ such that $\omega_x\notin\mathbb K$, then $|[\gb\omega]|=1$. Hence, it is natural to regard the quotient $\deg(\gb\omega)/|[\gb\omega]|$ as the inseparable degree of $\gb\omega$ over $\mathbb K$.
\end{rem}

We end this section by introducing some subsets of $\gb M_\mathbb F^{f_\mathbb K},\gb M_\mathbb K$, and $\gb P_\mathbb K$ which will be important later. We begin with
\begin{equation}
\cal P_\mathbb K^+=\{\gb\varpi\in\gb M_\mathbb K: \varpi_x=0 \text{ for all but finitely many }x\in X\}.
\end{equation}
To define the remaining subsets we need to equip $X$ with more structure. Namely, we suppose we are given a set $I$, a partition $\{X_i, i\in I\}$ of $X$, and a bijective map $\phi_i:\mathbb N\to X_i$  for each $i\in I$. Then, given $\gb\varpi\in\gb M_\mathbb K$, identify $\gb\varpi$ with the element of $(\mathbb K[[u]])^I$ whose value at $i\in I$ is given by
\begin{equation}
\gb\varpi_i(u) = 1+\sum_{r\in\mathbb N} \varpi_{\phi_i(r)}u^r.
\end{equation}
Henceforth, identify $\gb M_\mathbb K$ with the corresponding subset of $(\mathbb K[[u]])^I$ and regard each $\gb\varpi\in\gb M_\mathbb K$ as an $I$-tuple of formal power series with constant term 1. In particular, $\cal P_\mathbb K^+$ is identified with the corresponding set of $I$-tuples of polynomials. Moreover, the usual multiplicative monoid structure on $(\mathbb K[[u]])^I$ equips $\gb M_\mathbb K$ and $\cal P_\mathbb K^+$ with  monoid structures.

Two elements $\gb\omega,\gb\varpi\in\cal P_\mathbb K^+$ are said to be relatively prime if $\gb\omega_i(u)$ and $\gb\varpi_j(u)$ are relatively prime as elements of $\mathbb K[u]$ for all $i,j\in I$. Let $\cal P_\mathbb K$ be the abelian group corresponding to $\cal P_\mathbb K^+$. Since $\mathbb K[u]$ is a unique factorization domain, every element $\gb\varpi\in\cal P_\mathbb K$ can be  written as $\gb\varpi = \gb\omega\gb\pi^{-1}$ for some unique relatively prime elements $\gb\omega,\gb\pi\in \cal P_\mathbb K^+$. Define $\cal P_\mathbb F^+$ and $\cal P_\mathbb F$ by setting $\mathbb K=\mathbb F$ above. Observe that an element of the form $\gb\pi^{-1}\in(\cal P_\mathbb F^+)^{-1}$ can be identified with an element of $\gb M_\mathbb F$ by expanding the irreducible factors $(1-au)^{-1}, a\in\mathbb F$, as a geometric power series: $(1-au)^{-1}=\sum_r (au)^r$. In particular,  under this identification, $\cal P_\mathbb F$ is a subgroup of the multiplicative monoid $\gb M_\mathbb F$. Moreover, $\cal P_\mathbb F\subseteq \gb M_\mathbb F^{f_\mathbb K}$. Set
\begin{equation}
\cal P_{\mathbb F,\mathbb K}=\{[\gb\varpi]:\gb\varpi\in\cal P_\mathbb F\}\subseteq\gb P_\mathbb K \qquad\text{and}\qquad \cal P_{\mathbb F,\mathbb K}^+=\{[\gb\varpi]:\gb\varpi\in\cal P_\mathbb F^+\}.
\end{equation}

The next lemmas are trivially established.

\begin{lem}\label{l:galoisonP}
For any given structure $(X_i,\phi_i), i\in I$ as above, ${\rm Aut}(\mathbb F/\mathbb K)$ acts on $\gb M_\mathbb F$ by monoid  homomorphisms and, hence, by group homomorphisms on $\cal P_\mathbb F$.\hfill\qedsymbol
\end{lem}

Given $\gb\varpi,\gb\pi\in\gb M_\mathbb F$, let $\mathbb K(\gb\varpi,\gb\pi)$ be the field obtained from $\mathbb K$ by adjoining $\{\varpi_x,\pi_x:x\in X\}$. Observe that
\begin{equation}\label{e:chainfields}
\deg(\gb\varpi\gb\pi) =[\mathbb K(\gb\varpi\gb\pi):\mathbb K]\le [\mathbb K(\gb\varpi,\gb\pi):\mathbb K]\le [\mathbb K(\gb\varpi):\mathbb K][\mathbb K(\gb\pi):\mathbb K]= \deg(\gb\varpi)\deg(\gb\pi).
\end{equation}
It is easy to see that both inequalities above may be strict and that they are equalities if  $\deg(\gb\varpi)=1$.

\begin{lem}\label{l:tpfields}
Let $\gb\varpi,\gb\pi\in \gb M_\mathbb F^{f_\mathbb K}$ and $\varphi:\mathbb K(\gb\varpi)\otimes_\mathbb K\mathbb K(\gb\pi)\to \mathbb F$ be the $\mathbb K$-linear map induced by $a\otimes b\mapsto ab$. Then the image of $\varphi$ is $\mathbb K(\gb\varpi,\gb\pi)$ and $\varphi$ is injective iff $\deg(\gb\varpi)\deg(\gb\pi)=[\mathbb K(\gb\varpi,\gb\pi):\mathbb K]$.\hfill\qedsymbol
\end{lem}

\subsection{Simple Lie Algebras and Loop Algebras}\label{ss:simpleloop}\hfill\\

Let $I$ be the set of vertices of a finite-type connected Dynkin
diagram and let $\lie g$ be the associated simple Lie algebra over $\mathbb C$ with a fixed Cartan subalgebra $\lie h$.
Fix a set of positive roots $R^+$ and let
$$\lie n^\pm = \opl_{\alpha\in R^+}^{} \lie g_{\pm\alpha} \quad\text{where}\quad \lie g_{\pm\alpha} = \{x\in\lie g: [h,x]=\pm\alpha(h)x, \ \forall \ h\in\lie h\}.$$
The simple roots will be denoted by $\alpha_i$, the fundamental weights by $\omega_i$, while $Q,P,Q^+,P^+$ will denote the root and weight lattices with corresponding positive cones, respectively. We equip $\lie h^*$ with the partial order $\lambda\le \mu$ iff $\mu-\lambda\in Q^+$. We denote by $\cal W$ the Weyl group of $\lie g$ and let $w_0$ be its longest element.
Let $\langle\ ,\ \rangle$ be the bilinear form on $\lie h^*$ induced by the Killing form on $\lie g$ and, for a nonzero $\lambda\in\lie h^*$, set $\lambda^\vee = 2\lambda/\langle\lambda,\lambda\rangle$ and $d_\lambda = \frac{1}{2}\langle\lambda,\lambda\rangle$. Then $\{\alpha_i^\vee:i\in I\}$ is a set of simple roots of the simple Lie algebra $\lie g^\vee$ whose Dynking diagram is obtained from that of $\lie g$ by reversing the arrows and $R^\vee=\{\alpha^\vee:\alpha\in R\}$ is its root system, where $R=R^+\cup (-R^+)$.  Moreover, if $\alpha = \sum_i m_i\alpha_i$ and $\alpha^\vee=\sum_i m_i^\vee\alpha_i^\vee$, then
\begin{equation}\label{m_ivee}
m_i^\vee = \frac{d_{\alpha_i}}{d_\alpha}m_i.
\end{equation}

Given a Lie algebra $\lie a$ over a field $\mathbb L$, define the corresponding loop algebra as $\tlie a=\lie a\otimes_{\mathbb L}  \mathbb L[t,t^{-1}]$ with bracket given by $[x \otimes t^r,y \otimes t^s]=[x,y] \otimes t^{r+s}$. Clearly $\lie a\otimes 1$ is a subalgebra of $\tlie a$ isomorphic to $\lie a$ and, by abuse of notation, we will continue denoting its elements by $x$ instead of $x\otimes 1$. We have $\tlie g = \tlie n^-\oplus \tlie h\oplus \tlie n^+$ and $\tlie h$ is an abelian subalgebra.

Let $U(\lie a)$ denote the universal enveloping algebra of a Lie algebra $\lie a$. Then $U(\lie a)$ is a subalgebra of $U(\tlie a)$ and multiplication establishes isomorphisms of vector spaces
$$U(\lie g)\cong U(\lie n^-)\otimes U(\lie h)\otimes U(\lie n^+) \qquad\text{and}\qquad U(\tlie g)\cong U(\tlie n^-)\otimes U(\tlie h)\otimes U(\tlie n^+).$$
The assignments $\triangle: \lie a\to U(\lie a)\otimes U(\lie a), x\mapsto x\otimes 1+1\otimes x$, $S:\lie a\to \lie a,
x\mapsto -x$,  and $\epsilon: \lie a\to \mathbb L, x\mapsto 0$, can
be uniquely extended so that $U(\lie a)$ becomes a Hopf algebra with
comultiplication $\triangle$, antipode $S$, and  counit $\epsilon$.
We shall denote by $U(\lie a)^0$ the augmentation ideal, i.e., the
kernel of $\epsilon$.

Given $a\in\mathbb L^\times$, let ${\rm ev}_a:\tlie a\to\lie a$ be the evaluation map $x\otimes t^k\mapsto a^kx$. We also denote by ${\rm ev}_a$ the induced map $U(\tlie a)\to U(\lie a)$.

\subsection{Hyperalgebras}\label{ss:hyper}\hfill\\

 Given an associative algebra $A$ over a field of characteristic zero, $a\in A$, and $k\in\mathbb Z_+$, we set $a^{(k)}=\frac{a^k}{k!}$ and $\binom{a}{k} = \frac{a(a-1)\cdots(a-k+1)}{k!}\in A$.

Let $B=\{ x^\pm_{\alpha}, h_{\alpha_i}: \alpha\in R^+,i\in I\}$
be a Chevalley basis for $\lie g$, where $x^\pm_\alpha\in \lie
g_{\pm\alpha}$ and $h_\alpha=[x^+_\alpha,x^-_\alpha]$. Let
$x^\pm_{\alpha,r} = x^\pm_{\alpha}\otimes t^r$ and $h_{\alpha,r} =
h_\alpha\otimes t^r$. When $r=0$ we may write $x_{\alpha}^\pm$
and $h_{\alpha}$ in place of $x_{\alpha,0}^\pm$
and $h_{\alpha,0}$, respectively. Also, if $\alpha = \alpha_i$ we may simply write
$x_{i,r}^\pm, h_{i,r},x_i^+$, or $h_i$ accordingly. Notice that the set $\tilde B = \{x_{\alpha,r}^\pm,
h_{i,r}:\alpha\in R^+,i\in I,r\in\mathbb Z\}$ is a basis for $\tlie
g$, which we will refer to as a Chevalley basis for $\tlie g$. Define $\tlie g_{\mathbb Z}$ to be the $\mathbb Z$-span of
$\tilde B$ and observe that the $\mathbb Z$-span of $B$ is a Lie $\mathbb
Z$-subalgebra of $\tlie g_{\mathbb Z}$ which we denote by $\lie g_{\mathbb Z}$.

The following identity in $U(\tlie g)$ is easily deduced.
\begin{equation}\label{e:comutxh}
\binom{h_i}{l}(x_{\alpha,r}^{\pm})^{(k)} = (x_{\alpha,r}^{\pm})^{(k)}\binom{h_i\pm k\alpha(h_i)}{l}.
\end{equation}

Given $\alpha\in R^+, r\in\mathbb Z$, define elements $\Lambda_{\alpha,r}\in U(\tlie h)$ by the following equality of formal power series in the variable $u$:
\begin{equation}\label{e:Lambdadef}
\Lambda_\alpha^\pm(u) = \sum_{r=0}^\infty \Lambda_{\alpha,\pm r} u^r = \exp\left( - \sum_{s=1}^\infty \frac{h_{\alpha,\pm s}}{s} u^s\right).
\end{equation}
We may write $\Lambda_{i,r}$ in place of $\Lambda_{\alpha_i,r}$. It follows from \eqref{m_ivee} that, if $\alpha=\sum_i m_i\alpha_i\in R^+$, then $h_\alpha = \sum_i m_i^\vee h_i$ and
\begin{equation}\label{e:Lambda_alpha}
\Lambda_\alpha^\pm(u) = \prod_{i\in I} (\Lambda_{\alpha_i}^\pm(u))^{m_i^\vee}.
\end{equation}
We have (cf. \cite[Lemma 5.1]{garala}):
\begin{equation}\label{e:evLambda}
{\rm ev}_a(\Lambda_{\alpha,r}) = (-a)^r\binom{h_\alpha}{|r|}.
\end{equation}

For $k\in\mathbb N$, consider the endomorphism $\tau_k$
of $U(\tlie g)$ extending $t\mapsto t^k$ and set
$\Lambda_{\alpha,r;k}=\tau_k(\Lambda_{\alpha,r})$,
$\Lambda_{\alpha;k}^\pm(u) = \sum_{r=0}^\infty \Lambda_{\alpha,\pm
r;k} u^r $. Since $\binom{h_i}{k}$ is a polynomial in $h_i$ of
degree $k$,  the set $\{\prod_{i\in I}\binom{h_i}{k_i}:k_i\in\mathbb Z_+\}$ is a basis for $U(\lie h)$. Similarly, for every $r,k\in\mathbb N$, $\Lambda_{i,\pm r;k}$  is a polynomial in $h_{i,\pm sk}, 1\le s\le r$, whose leading term is $(-h_{i,\pm k})^{(r)}$. Finally, given an order on $\tilde B$ and a PBW monomial with respect to this order,  we construct an ordered monomial  in the elements
$$(x_{\alpha,r}^\pm)^{(k)},\ \ \Lambda_{i,r;k},\ \ \binom{h_i}{k}, \qquad r,k\in\mathbb Z, k>0, \alpha\in R^+, i\in I,$$
 using the correspondence discussed for the basis elements of $U(\tlie h)$, as well as, the obvious correspondence $(x_{\alpha,r}^\pm)^k \leftrightarrow  (x_{\alpha,r}^\pm)^{(k)}$. The set of monomials thus obtained is then a basis for $U(\tlie g)$, while the monomials involving only
$(x_{\alpha}^\pm)^{(k)},  \binom{h_i}{k}$ form a basis for
$U(\lie g)$. Let $U(\tlie g)_{\mathbb Z}$ (resp. $U(\lie g)_{\mathbb
Z}$) be the $\mathbb Z$-span of these monomials. The following theorem was proved in \cite{kosagz,garala}.

\begin{thm}\label{t:kosgar}
$U(\tlie g)_{\mathbb Z}$ (resp. $U(\lie g)_{\mathbb Z}$) is the $\mathbb Z$-subalgebra of $U(\tlie g)$ generated by $\{(x_{\alpha,r}^\pm)^{(k)}, \alpha\in R^+,r,k\in\mathbb Z, k\ge 0\}$ (resp. $\{(x_{\alpha}^\pm)^{(k)}, \alpha\in R^+, k\in\mathbb Z_+\}$).\hfill\qedsymbol
\end{thm}

For $\lie a\in\{ \lie g, \lie n^\pm, \lie h, \tlie g,\tlie n^\pm,
\tlie h\}$, let $U(\lie a)_{\mathbb Z}=U(\lie a)\cap U(\tlie g)_\mathbb Z$ and, given a field  $\mathbb L$, define the $\mathbb L$-hyperalgebra of $\lie a$ by
$$U(\lie a)_\mathbb L = U(\lie a)_{\mathbb Z}\otimes_{\mathbb Z} \mathbb L.$$
We refer to $U(\tlie g)_\mathbb L$ as the hyper loop algebra of $\lie g$ over $\mathbb L$.
The PBW theorem gives
$$U(\lie g)_{\mathbb L}=U(\lie n^-)_{\mathbb L}U(\lie h)_{\mathbb L}U(\lie n^+)_{\mathbb L}\qquad\text{and}\qquad U(\tlie g)_{\mathbb L}=U(\tlie n^-)_{\mathbb L}U(\tlie h)_{\mathbb L}U(\tlie n^+)_{\mathbb L}.$$
We will keep denoting by $x$ the image of an element $x\in U(\tlie g)_{\mathbb Z}$ in $U(\tlie g)_\mathbb L$.

The next relation was a key ingredient in the proof of Theorem \ref{t:kosgar} in the case of loop algebras and it also plays an essential role in the study of finite-dimensional representations. Given $\alpha\in R^+, s\in\mathbb Z$, set
$$X^-_{\alpha;s,\pm}(u) = \sum_{r\ge 1} x^-_{\alpha,\pm(r+s)}u^r.$$
We have:
\begin{gather}\label{e:basicrel}
(x^+_{\alpha,\mp s})^{(l)}(x^-_{\alpha,\pm(s+1)})^{(k)} \in (-1)^l \left((X_{\alpha;s,\pm}^-(u))^{(k-l)}\Lambda_{\alpha}^{\pm}(u)\right)_k + U(\tlie g)_\mathbb Z U(\tlie n^+)_\mathbb Z^0,
\end{gather}
where $k\ge l\ge 1$ and the subindex $k$  means the coefficient of $u^k$ of the above power series (cf. \cite[Lemma 1.3]{jm}).

Observe that
\begin{equation}\label{e:binid}
(x_{\alpha,r}^\pm)^{(k)}(x_{\alpha,r}^\pm)^{(l)} = \binom{k+l}{k}(x_{\alpha,r}^\pm)^{(k+l)}.
\end{equation}
Suppose $\mathbb L$ is a field of characteristic $p>0$. It follows that
$((x_{\alpha,r}^\pm)^{(k)})^p=0$ for all $\alpha\in R^+,r\in\mathbb Z$, and $k\in\mathbb N$. Moreover, $U(\tlie g)_\mathbb L$ is generated, as an $\mathbb L$-associative algebra, by the elements $(x_{\alpha,r}^\pm)^{(p^k)}, \alpha\in R^+, r\in\mathbb Z, k\in\mathbb Z_+$ and $U(\lie h)_\mathbb L$ is generated by $\binom{h_i}{p^k}, i\in I,k\in\mathbb Z_+$.

Let $\gb\Lambda_\mathbb Z$ be the subalgebra of  $U(\tlie h)_\mathbb Z$ generated by the elements $\Lambda_{i,r;k}, i\in I,r\in\mathbb Z\backslash\{0\},k\in\mathbb N$, and observe that Theorem \ref{t:kosgar} and the PBW theorem imply that
$$U(\tlie h)_\mathbb Z\cong U(\lie h)_\mathbb Z\otimes_\mathbb Z\gb\Lambda_\mathbb Z.$$
In particular, the same holds with $\mathbb Z$ replaced by a field $\mathbb L$, where $\gb\Lambda_\mathbb L=\gb\Lambda_\mathbb Z\otimes_\mathbb Z\mathbb L$. The next proposition plays a prominent role in the development of the main results of this paper.

\begin{prop}\label{p:Lambdapoly}
Let $\phi: I\times(\mathbb Z\backslash\{0\})\to \mathbb N$ be a bijection and $\mathbb Z[u_j:j\in\mathbb N]$ the $\mathbb Z$-algebra of polynomials in the variables $u_j, j\in\mathbb N$. The assignment $\Lambda_{i,r}\mapsto u_{\phi(i,r)}$ establishes an isomorphism from $\gb\Lambda_\mathbb Z$ onto $\mathbb Z[u_j:j\in\mathbb N]$. In particular, the same holds with a field $\mathbb L$ in place of $\mathbb Z$.
\end{prop}

\begin{proof}
Since $\tlie h$ is commutative and the elements $h_{i,r}, i\in I,r\in\mathbb Z$ form a basis of $\tlie h$, it follows from the PBW theorem that $U(\tlie h)$ is the polynomial algebra in the variables $h_{i,r}$. Taking $\log$ of \eqref{e:Lambdadef} we get
$$\sum_{s=1}^\infty \frac{h_{\alpha,\pm s}}{s}\ u^s = \sum_{k\ge 1} \frac{1}{k}\left(-\sum_{r=1}^\infty \Lambda_{\alpha,\pm r}\ u^r\right)^k,$$
and one easily deduces that  $U(\tlie h)$ is also the polynomial algebra in the variables $h_i,\Lambda_{i,r}, i\in I,r\in\mathbb Z-\{0\}$. Since the elements $\Lambda_{i,r;k}$ belong to a $\mathbb Z$-basis of $\gb\Lambda_\mathbb Z$ by Theorem \ref{t:kosgar}, it suffices to show that each $\Lambda_{i,r;k}$ belongs to $\mathbb Z[\Lambda_{i,s}:s\in\mathbb Z]$. This is known to be true  (cf. \cite[Lemma 1.5]{jm}).
\end{proof}

For each $i\in I$ set
\begin{equation}
\phi_i:\mathbb N\to U(\tlie h)_\mathbb K, \qquad \phi_i(r) = \Lambda_{i,r},
\end{equation}
and define
\begin{equation}\label{e:boldlambda}
\gb\Lambda_i = \{\phi_i(r): r\in\mathbb N\}\quad\text{and}\quad \gb\Lambda=\bigcup_{i\in I}\gb\Lambda_i.
\end{equation}
From now on, we use the notation $\mathbb F[\gb\Lambda], \mathbb K[\gb\Lambda], \gb P_\mathbb K,\cal P_\mathbb F^+, \cal P_{\mathbb F,\mathbb K}$, etc., introduced  in Section \ref{ss:irpa}  with $X=\gb\Lambda$ and the above choice of partition and maps $\phi_i$.
The set $\cal P_{\mathbb F,\mathbb K}$ is called the quasi-$\ell$-weight lattice of $U(\tlie g)_\mathbb K$ and its elements will be referred to as the (integral) quasi-$\ell$-weights. Similarly, the sets $\cal P_\mathbb K, \cal P_{\mathbb F,\mathbb K}^+$, and $\cal P_\mathbb K^+$ are called  the $\ell$-weight-lattice, the set of dominant quasi-$\ell$-weights, and the set of dominant $\ell$-weights, respectively.

Since the Hopf algebra structure on $U(\tlie g)$ preserves the $\mathbb Z$-forms $U(\lie a)_\mathbb Z$, where $\lie a\in\{ \lie g, \lie n^\pm, \lie h, \tlie g,\tlie n^\pm,
\tlie h\}$, it induces a Hopf algebra structure on $U(\tlie g)_\mathbb L$ with counit given by $\epsilon((x_{\alpha,r}^\pm)^{(k)})=0, \epsilon(a)=a, a\in\mathbb L$, and comultiplication given by
\begin{equation}\label{e:comnil}
\triangle\left((x_{\alpha,r}^\pm)^{(k)}\right)= \sum_{l+m=k} (x_{\alpha,r}^\pm)^{(l)}\otimes (x_{\alpha,r}^\pm)^{(m)},
\end{equation}
\begin{equation}\label{e:comcart}
\triangle\left(\binom{h_i}{k}\right)= \sum_{l+m=k}\binom{h_i}{l}\otimes \binom{h_i}{m}, \quad\text{and}\quad
\triangle(\Lambda_{\alpha,\pm k}) = \sum_{l+m=k} \Lambda_{\alpha,\pm l}\otimes \Lambda_{\alpha,\pm m}.
\end{equation}
The antipode on the basis of $U(\tlie h)_\mathbb L$ is determined by the following equalities of formal power series
\begin{equation}\label{e:antipode}
S(\Lambda^\pm_{\alpha;k}(u)) = (\Lambda^\pm_{\alpha;k}(u))^{-1} \qquad\text{and}\qquad S(H_\alpha(u)) = (H_\alpha(u))^{-1},
\end{equation}
where
\begin{equation}\label{e:H}
H_\alpha(u) = \sum_{k\ge 0} \binom{h_\alpha}{k} u^k = {\rm ev}_{-1}(\Lambda^+_\alpha(u)).
\end{equation}

We will use the following notation below. For each $r\in\mathbb Z, \alpha\in R^+$, let $U(\tlie g_{\alpha,r})_\mathbb K$ be the subalgebra of $U(\tlie g)_\mathbb K$ generated by the elements $(x_{\alpha,\pm r}^{\pm})^{(k)},  k\in\mathbb Z_+$. $U(\tlie g_{\alpha,r})_\mathbb K$ is clearly isomorphic to $U(\lie{sl}_2)_\mathbb K$.

The next theorem follows from  \cite[Lemma 1.3]{mats} and \cite[Lemma 9.5]{cpr}.

\begin{thm}\label{t:frobenius}
Let $\mathbb K$ be a field of characteristic $p>0$. There exists a unique Hopf algebra homomorfism $\tilde\Phi:U(\tlie g)_\mathbb K\to U(\tlie g)_\mathbb K$ sending $(x_{\alpha,r}^\pm)^{(p^k)}$ to $(x_{\alpha,r}^\pm)^{(p^{k-1})}$, where the later is understood to be zero when $k=0$.\hfill\qedsymbol
\end{thm}

The map $\tilde\Phi$ is called the (arithmetic) Frobenius homomorphism. We shall denote  by $V^{\tilde\Phi}$ the pull-back of a $U(\tlie g)_\mathbb K$-module $V$ by $\tilde\Phi$.

\subsection{Finite-Dimensional Representations of $U(\lie g)_\mathbb K$}\label{ss:repg}\hfill\\

We now review the finite-dimensional representation theory of $U(\lie g)_\mathbb K$. When $\mathbb K$ is of characteristic zero the results reviewed are classical and can be found in \cite{humb} for instance. In positive  characteristic the literature is more vast in the context of algebraic groups. See however \cite[Section 2]{jm} for a more detailed review in the present context and a few references.

Let $V$ be a $U(\lie h)_\mathbb K$-module. A nonzero vector $v\in V$ is
called a weight vector if there exists $\gb z = (z_{i,k}),
z_{i,k}\in \mathbb K, i\in I,k\in\mathbb Z_+$, called the weight of $v$, such that  $\binom{h_i}{p^k}v
= z_{i,k}v$.  We say that $\gb z$ is  integral (resp. dominant integral) if $z_{i,k} =
\binom{\mu(h_i)}{p^k}$ for some $\mu\in P$ (resp. $\mu\in P^+$). In that case we identify $\gb z$ with $\mu$ and say that $v$ has weight $\mu$. Let $V_{\gb z}$ be the subspace of $V$ spanned by the weight vectors of weight $\gb z$.
The nontrivial subspaces $V_{\gb z}$ are called the weight spaces of $V$ and $\gb z$ is said to be a weight of $V$. $V$ is said to be a weight module if it is the direct sum of its weight spaces.

If $V$ is a $U(\lie g)_\mathbb K$-module and $v$ is a weight vector such that
$(x_\alpha^+)^{(k)}v=0$ for all $\alpha\in R^+, k\in\mathbb N$, then $v$ is
said to be  a highest-weight  vector. If $V$ is generated  by a
highest-weight vector, $V$ is called a highest-weight module. It clearly follows from \eqref{e:comutxh} that every highest-weight module is a weight module.

It turns out that, if $V$ is a finite-dimensional $U(\lie g)_\mathbb K$-module, the eigenvalues of the elements $\binom{h_i}{p^k}$ lie in the prime field of $\mathbb K$. In particular, since the elements $\binom{h_i}{p^k}$ commute, we can decompose any finite-dimensional representation $V$ of $U(\lie g)_\mathbb K$ in a direct
sum of generalized eigenspaces for the action of $U(\lie h)_\mathbb K$. Moreover, it can be shown that $V_\gb z\ne 0$ only if $\gb z$ is integral and that $V$ is a weight module. From now on we denote $V_\gb z$ by $V_\mu$ where $\mu$ is the unique element of $P$ such that $\binom{\mu(h_i)}{p^k} = z_{i,k}$.

It follows from \eqref{e:comutxh} that if $v$ has weight $\mu$ then
$(x_\alpha^\pm)^{(k)}v$ is either zero or has weight $\mu\pm
k\alpha$. Hence, if $v$ is a highest-weight vector of weight $\lambda$ of a
highest-weight module $V$, we have $\dim(V_{\lambda})=1$ and
$V_{\mu}\ne 0$ only if $\mu\le \lambda$. In particular, every highest-weight $U(\lie g)_\mathbb K$-module has a unique irreducible quotient.

\begin{defn}
Given $\lambda\in P^+$,  let $W_\mathbb K(\lambda)$ be the $U(\lie
g)_\mathbb K$-module generated by a vector $v$ satisfying
\begin{gather}
(x_{\alpha}^+)^{(p^k)}v = 0, \quad \binom{h_i}{p^k}v = \binom{\lambda(h_i)}{p^k}v,\quad (x_{\alpha}^-)^{(l)}v = 0, \quad \forall\ \alpha\in R^+, i\in I, k,l\in\mathbb Z_+, l>\lambda(h_\alpha).
\end{gather}
\end{defn}

Recall that $w_0$ denotes the longest element of the Weyl group of $\lie g$.

{\samepage
\begin{thm}\label{t:cig}\hfill
\begin{enumerate}
\item If $V$ is a finite-dimensional highest-weight $U(\lie g)_\mathbb K$-module,
then its highest weight is dominant integral.
\item Every irreducible finite-dimensional $U(\lie g)_\mathbb K$-module is highest-weight.
\item For every $\lambda\in P^+$, $W_\mathbb K(\lambda)$ is the universal finite-dimensional $U(\lie g)_\mathbb K$-module of highest weight $\lambda$.
\item $W_\mathbb K(\lambda)_\mu\ne 0$ only if $w_0\lambda\le\mu\le\lambda$.
\item If $\mathbb K$ has characteristic zero, $W_\mathbb K(\lambda)$ is irreducible.\hfill\qedsymbol
\end{enumerate}
\end{thm}}

$W_\mathbb K(\lambda)$ is called the Weyl module with highest weight $\lambda$. Denote by $V_\mathbb K(\lambda)$ the irreducible quotient of $W_\mathbb K(\lambda)$.

\begin{cor}\label{c:cig}
The assignment $\lambda\mapsto V_\mathbb K(\lambda)$ induces a bijection from $P^+$ to the set of isomorphism classes of irreducible $U(\lie g)_\mathbb K$-modules. \hfill\qedsymbol
\end{cor}

\begin{rem}
Let $\mathbb P$ be the prime field of $\mathbb K$ and consider the functor sending a $U(\lie g)_\mathbb P$-module $V$ to the $U(\lie g)_\mathbb K$-module $V^\mathbb K=V\otimes_\mathbb P\mathbb K$. It is not difficult to see that $V_\mathbb K(\lambda)\cong V_\mathbb P(\lambda)\otimes_\mathbb P\mathbb K$ (similarly $W_\mathbb K(\lambda)\cong W_\mathbb P(\lambda)\otimes_\mathbb P\mathbb K$).
In particular, if $V$ is a finite-dimensional $U(\lie g)_\mathbb P$-module, we have $\ch(V) = \ch(V^\mathbb K)$, where $\ch$ is the formal character.

\end{rem}

\section{Finite-Dimensional Representations of Hyper Loop Algebras}

\subsection{Quasi-$\ell$-Weight Modules}\label{ss:qewm}\hfill\\

\begin{defn}
Let $V$ be a $U(\tlie h)_\mathbb K$-module and $W$ be a finite-length $\mathbb K[\gb\Lambda]$-submodule of $V$. $W$ is said to be of quasi-$\ell$-weight $[\gb\varpi]\in \gb M_{\mathbb F,\mathbb K}$ if all irreducible constituents of $W$ are isomorphic to $\cal K(\gb\varpi)$. If $\gb\varpi\in\gb M_\mathbb K$ we may simply say that $W$ is of $\ell$-weight  $\gb\varpi$ (notice that in this case $|[\gb\varpi]|=1$ and $\dim(\cal K(\gb\varpi))=1).$ Given $\gb\varpi\in\gb M_\mathbb F$, let $V_\gb\varpi$ be the sum of all finite-length $\mathbb K[\gb\Lambda]$-submodules of $V$ of (quasi)-$\ell$-weight $[\gb\varpi]$. If $V_\gb\varpi\ne 0$, the conjugacy class $[\gb\varpi]$ is said to be a (quasi)-$\ell$-weight of $V$ and $V_\gb\varpi$ is called the $[\gb\varpi]$-(quasi)-$\ell$-weight space of $V$.
$V$ is said to be a (quasi)-$\ell$-weight module if it is a weight-module and
$$V=\opl_{[\gb\varpi]\in\gb M_{\mathbb F,\mathbb K}}^{} V_\gb\varpi.$$

A $U(\tlie g)_\mathbb K$-module $V$ is said to be a  highest-(quasi)-$\ell$-weight module if there exist  $[\gb\omega]\in\gb M_{\mathbb F,\mathbb K}$ and a nontrivial subspace $\cal V$ of $V_\gb\omega$ such that $V=U(\tlie g)_\mathbb K\cal V$, $\cal V$ is an irreducible $U(\tlie h)_\mathbb K$-submodule of $V$ of finite length as a $\mathbb K[\gb\Lambda]$-module,  and $U(\tlie n^+)^0_\mathbb K\cal V=0$. In that case, the subspace $\cal V$ is called the highest-(quasi)-$\ell$-weight space of $V$ and $[\gb\omega]$ is said to be the highest (quasi)-$\ell$-weight of $V$.
\end{defn}

The next proposition is immediate from Corollary \ref{c:irpa}.

\begin{prop}
Every finite-dimensional $U(\tlie h)_\mathbb K$-module is a quasi-$\ell$-weight module.\hfill\qedsymbol
\end{prop}

\begin{rem}
We shall see examples of finite-dimensional  highest-$\ell$-weight modules which are not $\ell$-weight modules. Notice that we do not require a highest-(quasi)-$\ell$-weight module to be a quasi-$\ell$-weight module. In fact, it is not clear if the former implies the later in general.
\end{rem}

Clearly the finite-dimensional quasi-$\ell$-weight spaces of a quasi-$\ell$-weight module are weight modules themselves. Let $V$ be a highest-quasi-$\ell$-weight module with finite-dimensional highest quasi-$\ell$-weight space $\cal V$ and write $\cal V=\opl_{\gb z}^{} \cal V_{\gb z}$. Then, each $\cal V_{\gb z}$ is a $U(\tlie h)_{\mathbb K}$-submodule of $\cal V$. Since $\cal V$ is an irreducible $U(\tlie h)_{\mathbb K}$-module, it follows that $\cal V\subseteq V_{\gb z}$ for some weight $\gb z$ of $V$. Moreover, since $V=U(\tlie n^-)\cal V$, it follows that $\cal V= V_{\gb z}$ and $\gb z$ is the maximal weight of $V$.

\begin{lem}
Let $V$ be a finite-dimensional highest-quasi-$\ell$-weight $U(\tlie g)_\mathbb K$-module and let $\cal V=V_{\gb z}$ be its highest-quasi-$\ell$-weight space. Then, $\gb z=\lambda$ for some $\lambda\in P^+$ and $(x_{\alpha,s}^-)^{(k)}\cal V =0$ for all $\alpha\in R^+, k>\lambda(h_\alpha)$ and all $s\in\mathbb Z$.
\end{lem}

\begin{proof}
Identical to the proof of the first statement of \cite[Proposition 3.1]{jm}.
\end{proof}

The next theorem, which is a generalization of \cite[Proposition 1.1]{cpweyl} and of \cite[Proposition 3.1]{jm} is the first key step justifying why we called $\cal P_{\mathbb F,\mathbb K}^+$ the set of dominant quasi-$\ell$-weights.

\begin{thm}\label{t:ellhwrel}
Let $V$ be a highest-quasi-$\ell$-weight $U(\tlie g)_\mathbb K$-module with finite-dimensional highest-quasi-$\ell$-weight space $\cal V=V_\lambda, \lambda\in P^+$, such that $(x_{\alpha}^-)^{(k)}\cal V =0$ for all $\alpha\in R^+,k>\lambda(h_\alpha)$. Then:
\begin{enumerate}
\item Every vector $v\in V$ generates a finite-dimensional $U(\lie g)_\mathbb K$-submodule of $V$. In particular, $V_\mu\ne 0$ only if $w_0\lambda\le\mu\le\lambda$.
\item For every $i\in I$ and $r>\lambda(h_i)$, $\Lambda_{i,\pm r}\cal V=0$. Also, $\Lambda_{i,\pm\lambda(h_i)}$ acts bijectively on $\cal V$.
\item Given $i\in I$ and $r=1,\cdots, \lambda(h_i)$, there exist a polynomial $f_{i,r}\in\mathbb K[u_1,\dots,u_{\lambda(h_i)}]$ such that
$\Lambda_{i,-r}v = f_{i,r}(\Lambda_{i,1},\cdots, \Lambda_{i,\lambda(h_i)})v$ for all $v\in\cal V$. Moreover, the definition of each polynomial $f_{i,r}$ depends only on $\lie g$ and on the isomorphism class of $\cal V$ as a $\mathbb K[\gb\Lambda]$-module (but not on $V$).
\end{enumerate}
\end{thm}

\begin{proof}
Part (a) is standard. Notice that it follows from (a) that $(x_{\alpha,s}^-)^{(k)}\cal V =0$ for all $\alpha\in R^+,k>\lambda(h_\alpha)$ and all $s\in\mathbb Z$.

The arguments used to prove parts (b) and (c) are analogous to those used in the proof of \cite[Proposition 3.1]{jm} with some extra care.  Fix a nonzero $v\in \cal V$. Setting $\alpha=\alpha_i,s=0, l=k=r$ in \eqref{e:basicrel}  we get $\Lambda_{i,\pm r}v=0$ for $r>\lambda(h_i)$. By Schur's Lemma, in order to show that $\Lambda_{i,\pm\lambda(h_i)}$ acts bijectively on $\cal V$, it suffices to show that $\Lambda_{i,\pm\lambda(h_i)}v\ne 0$. Set $r=\lambda(h_i)$ and observe that part (a) implies  that $\lambda-(r+m)\alpha_i$ is not a weight of $V$ if $m>0$. Therefore $(x^-_i)^{(m)}(x_{i,\pm 1}^-)^{(r)}v=0$ for all $m\in\mathbb N$. On the other hand, by considering the subalgebra $U(\tlie g_{\alpha_i,\mp 1})_\mathbb K$, we see that $(x_{i,\pm 1}^-)^{(r)}v\ne 0$. It follows that $(x_{i,\pm 1}^-)^{(r)}v$ generates a lowest weight finite-dimensional representation of $U(\tlie g_{\alpha_i,0})_\mathbb K$ and, in particular, $0\ne (x_{i}^+)^{(r)}(x_{i,\pm 1}^-)^{(r)}v$. But the later is equal to $\Lambda_{i,\pm r}v$ by \eqref{e:basicrel}. Thus, (b) is proved.

By letting $m_i=\lambda(h_i)$ and setting  $\alpha=\alpha_i, s=0,l=m_i$ and $k=l+r\, (r>0)$ in \eqref{e:basicrel} we get
$$(x_{i,1}^-)^{(r)}\Lambda_{i,m_i}v + \sum_{j=1}^{m_i} Y_j\Lambda_{i,m_i-j}v=0,$$
where  $Y_j$ is the sum of the monomials $(x_{i,1}^-)^{(k_1)}\cdots (x_{i,m_i+1}^-)^{(k_{m_i+1})}$ such that $\sum_n k_n = r$ and $\sum_n nk_n = r+j$. Now, since $-r<r+j-2r<m_i$, it is not difficult to see that $(x_{i,-2}^+)^{(r)}Y_j \in U(\tlie g)_\mathbb FU(\tlie n^+)_\mathbb K^0 + H_j$, where $H_j$ is a linear combination of monomials of the form $\Lambda_{i,r_1}\cdots \Lambda_{i,r_m}$ such that $-r<r_n< m_i$. Moreover, $(x_{i,-2}^+)^{(r)}(x_{i,1}^-)^{(r)} \in (-1)^r\Lambda_{i,-r} + U(\tlie g)_\mathbb FU(\tlie n^+)_\mathbb K^0$ by \eqref{e:basicrel}. Hence,
$$0=(x_{i,-2}^+)^{(r)}\left((x_{i,1}^-)^{(r)}\Lambda_{i,m_i}v + \sum_{j=1}^{m_i} Y_j\Lambda_{i,m_i-j}v\right)= (-1)^r\Lambda_{i,-r}\Lambda_{i,m_i}v + \sum_{j=1}^{m_i} H_j\Lambda_{i,m_i-j}v.$$
Since $\cal V$ is finite-dimensional, the inverse of the action of $\Lambda_{i,m_i}$ on $\cal V$ is of the form $g_i(\Lambda_{i,m_i})$ for some $g_i\in\mathbb K[u]$ and we get
$$\Lambda_{i,-r}v = (-1)^{r+1}g_i(\Lambda_{i,m_i})\sum_{j=1}^{m_i} H_j\Lambda_{i,m_i-j}v.$$
An easy inductive argument on $r=1,\cdots,m_i$ completes the proof.
\end{proof}

The next corollary is now immediate.

\begin{cor}\label{c:ellhwrel}
Let $V, \cal V$, and $\lambda$ be as in the above theorem. Then $\cal V\cong \cal K(\gb\omega)$ as a $\mathbb K[\gb\Lambda]$-module for some $\gb\omega\in \cal P_\mathbb F^+$. Moreover, $\lambda = \sum_{i\in I}\deg(\gb\omega_i(u))\omega_i$.\hfill\qedsymbol
\end{cor}

\begin{defn}
Given $\gb\omega\in\cal P_\mathbb F^+$, define the weight of $\gb\omega$ as
$$\wt(\gb\omega)=\sum_{i\in I}\deg(\gb\omega_i(u))\omega_i.$$
Let $\wt:\cal P_\mathbb F\to P$ be the unique group homomorphism given by the above formula on $\cal P_\mathbb F^+$.
\end{defn}

\subsection{Weyl Modules}\label{ss:wm}\hfill\\

The next theorem may be regarded as part (d) of Theorem \ref{t:ellhwrel}.

\begin{thm}\label{t:wm}
Let $\gb\omega\in\cal P_\mathbb F^+$ and suppose $V$ is a highest-quasi-$\ell$-weight $U(\tlie g)_\mathbb K$-module with highest-quasi-$\ell$-weight space $\cal V$  satisfying $\cal V=V_\lambda$ with $\lambda = \wt(\gb\omega)$, $(x_{\alpha}^-)^{(l)}v=0$ for all $v\in \cal V, \alpha\in R^+,l\in\mathbb N, l>\lambda(h_\alpha)$.
Then $V$ is finite-dimensional.
\end{thm}

\begin{proof}
Since $\cal K(\gb\omega)$ (and hence $\cal V$) is finite-dimensional, it follows from Corollary \ref{c:ellhwrel} that $\cal V\cong \cal K(\gb\omega)$ as a $\mathbb K[\gb\Lambda]$-module.  Although the proof of the present theorem is  almost an exact copy of the proof of \cite[Theorem 3.11]{jm}, we find it convenient to write it down since the notation of this proof will be used in later proofs. Also this will make it clearer where the small differences in comparison with \cite[Theorem 3.11]{jm} lie.
It suffices to prove that $V$ is spanned by the elements $$(x_{\beta_1,s_1}^-)^{(k_1)}\cdots (x_{\beta_n,s_n}^-)^{(k_n)}v,$$ with $v\in\cal V, n,s_j,k_j\in\mathbb Z_+, \beta_j\in R^+$ such that $s_j<\lambda(h_{\beta_j})$ and $\sum_j k_j\beta_j\le \lambda-w_0\lambda$. The above elements  with no restriction on $s_j,\beta_j,k_j$ span $V$ by the PBW theorem and the condition $\sum_j k_j\beta_j\le \lambda-w_0\lambda$ is immediate from Theorem \ref{t:ellhwrel}(a).

Let $\cal R = R^+\times \mathbb Z\times\mathbb Z_+$, $\Xi$  the set of functions $\xi:\mathbb N\to \cal R$ given by $j\mapsto \xi_j=(\beta_j,s_j,k_j)$ such that $k_j=0$ for all $j$ sufficiently large, and $\Xi'$  the subset of $\Xi$ consisting of the elements $\xi$ such that $0\le s_j<\lambda(h_{\beta_j})$. Given $\xi\in\Xi$ and $v\in\cal V$ we define an element $v_\xi\in V$ as follows: if $k_j=0$ for $j>n$, then $v_\xi = (x_{\beta_1,s_1}^-)^{(k_1)}\cdots (x_{\beta_n,s_n}^-)^{(k_n)}v$. Define the degree of $\xi$ to be $d(\xi) = \sum_j k_j$ and the maximal exponent of $\xi$ to be $e(\xi)= \max\{k_j\}$.
Then $e(\xi)\le d(\xi)$ and $d(\xi)\ne 0$ implies $e(\xi)\ne 0$. Assume that  $d(\xi)>0$ (otherwise there is nothing to prove), let $\Xi_{d,e}$ be the subset of $\Xi$ consisting of all $\xi$ satisfying $d(\xi)=d$ and $e(\xi)=e$, and set $\Xi_d=\bigcup\limits_{1\le e\le d} \Xi_{d,e}$.

For each fixed $v\in\cal V$ we prove by induction on $d$ and sub-induction on $e$ that if
$\xi\in\Xi_{d,e}$ is such that there exists $j$ with either
$s_{j}<0$ or $s_j\ge \lambda(h_{\beta_{j}})$, then $v_\xi$ is in the
span of vectors of the form $w_\zeta$ with $w\in\cal V$ and $\zeta\in\Xi'$. More precisely,
given $0<e\le d\in\mathbb N$, we assume, by induction hypothesis,
that this statement is true for every $\xi$ which belongs either to
$\Xi_{d,e'}$ with $e'<e$ or to $\Xi_{d'}$ with $d'<d$.
Now \eqref{e:basicrel} and the condition $(x_{\alpha}^-)^{(k)}v=0$ for sufficiently large $k$ imply
\begin{equation}\label{basicrelv}
\left((X_{\beta;r,+}^-(u))^{(k-l)}\Lambda_{\beta}^+(u)\right)_kw = 0 \qquad \forall\ \beta\in R^+, k,l,r\in\mathbb Z, k>\lambda(h_\beta), 1\le l\le k, w\in\cal V.
\end{equation}
The proof is split in two cases according to whether $e=d$ or $e<d$ and the later one is proved exactly as in \cite{jm}.
When $e=d$ we have $v_\xi= (x_{\beta,s}^-)^{(e)}v$ for some $\beta\in R^+$ and $s\in\mathbb Z$. Set $l=e\lambda(h_{\beta})$ and $k=l+e$ in \eqref{basicrelv} to obtain
\begin{equation}\label{basicrelv2}
\sum_{n=0}^{\lambda(h_{\beta})} (x_{\beta,r+1+n}^-)^{(e)}\Lambda_{\beta,l-en}w + \text{ other terms} = 0,
\end{equation}
where the other terms belong to the span of elements $u_{\xi'}$ with $u\in\cal V$ and $\xi'\in \Xi_{e,e'}$ for $e'<e$. We consider the cases $s\ge l$ and $s<0$  separately  and prove the statement by  a further induction on $s$ and $|s|$, respectively. If $s\ge l$ this is easily done by setting $r=s-1-\lambda(h_\beta)$ and $w=v$ in \eqref{basicrelv2}. For $s<0$, let $n_0=\min\{n\in\mathbb Z_+:n\le\lambda(h_\beta)$ and $\Lambda_{\beta,l-en}\cal V\ne 0\}$. Then \eqref{basicrelv2} can be written as
\begin{equation}\label{basicrelv3}
\sum_{n=n_0}^{\lambda(h_{\beta})} (x_{\beta,r+1+n}^-)^{(e)}\Lambda_{\beta,l-en}w + \text{ other terms} = 0,
\end{equation}
Since $\Lambda_{\beta,l-en_0}$ is invertible on $\cal V$, the case $s<0$ is dealt with by setting $r=s-1$ and $w=(\Lambda_{\beta,l-en_0})^{-1}v$ in \eqref{basicrelv3}.
\end{proof}

\begin{defn}\label{d:wm}

Given $\gb\omega\in\cal P_\mathbb F^+$, let $\cal M_\gb\omega$ be the maximal ideal of $\mathbb K[\gb\Lambda]$ such that $\cal K(\gb\omega)\cong \mathbb K[\gb\Lambda]/\cal M_\gb\omega$. Also, let $\lambda=\wt(\gb\omega)$ and let $f_{i,r}$ be the polynomials given by Theorem \ref{t:ellhwrel} if $1\le r\le \lambda(h_i)$ and $f_{i,r}=0$ if $r>\lambda(h_i)$. The Weyl module $W_\mathbb K(\gb\omega)$ is the quotient of $U(\tlie g)_\mathbb K$ by the left ideal generated by
\begin{gather*}
U(\tlie n^+)_\mathbb K^0,\qquad  (x_{\alpha}^-)^{(l)} \quad \text{for all}\quad \alpha\in R^+,  l>\wt(\gb\omega)(h_\alpha),\\
\binom{h_i}{k}-\binom{\lambda(h_i)}{k}, \quad \cal M_\gb\omega, \quad \Lambda_{i,-r}-f_{i,r}(\Lambda_{i,1},\dots,\Lambda_{i,\lambda(h_i)}) \quad  \text{for all}\quad i\in I,k,r\in\mathbb N.
\end{gather*}
\end{defn}

It follows from Theorem \ref{t:wm} that $W_\mathbb K(\gb\omega)$ is finite-dimensional. Then Theorem \ref{t:ellhwrel} implies that $W_\mathbb K(\gb\omega)$ is the universal finite-dimensional $U(\tlie g)_\mathbb K$-module of highest quasi-$\ell$-weight $[\gb\omega]$ and, if $v$ is the image of $1$ in $W_\mathbb K(\gb\omega)$, then $\cal V:=U(\tlie h)_\mathbb Kv$ is isomorphic to $\cal K(\gb\omega)$ as a $\mathbb K[\gb\Lambda]$-module.
Moreover, quite clearly any proper submodule of $W_\mathbb K(\gb\omega)$ does not intersect $\cal V$. Hence, $W_\mathbb K(\gb\omega)$ has a unique maximal proper submodule and, therefore, a unique irreducible quotient which we denote by $V_\mathbb K(\gb\omega)$.
On the other hand, since a finite-dimensional $U(\tlie g)_\mathbb K$-module is a weight-module, it follows from \eqref{e:comutxh} that every finite-dimensional irreducible $U(\tlie g)_\mathbb K$ module is highest-quasi-$\ell$-weight and, therefore, a quotient of $W_\mathbb K(\gb\omega)$ for some $\gb\omega\in\cal P_\mathbb  F^+$. This proves:

\begin{cor}\label{c:citg}
$W_\mathbb K(\gb\omega)$ is the universal finite-dimensional highest-quasi-$\ell$-weight module of highest quasi-$\ell$-weight $[\gb\omega]$ and the assignment $\gb\omega\mapsto V_\mathbb K(\gb\omega)$ induces a bijection from $\cal P_{\mathbb F,\mathbb K}^+$  to the set  of isomorphism classes of finite-dimensional irreducible $U(\tlie g)_\mathbb K$-modules.\hfill\qedsymbol
\end{cor}

\begin{rem}
Since the set $\cal P_{\mathbb F,\mathbb K}^+$ is in bijective correspondence with the set of maximal ideals of $\mathbb K[\gb\Lambda]$ containing all elements $\Lambda_{i,r}$ for sufficiently large $r$, the set of isomorphism classes of irreducible finite-dimensional $U(\tlie g)_\mathbb K$-modules is in bijective correspondence with such ideals. We prefer to state the above results in terms of the conjugacy classes $[\gb\omega]$ since then the connection with the existing literature on finite-dimensional representations of classical, hyper, and quantum loop algebras is more evident.
\end{rem}

\subsection{Base Change}\label{ss:forms}\hfill\\

Let $A$ be an $\mathbb F$-associative algebra and $B$ a $\mathbb K$-subalgebra of $A$. We recall that a $B$-form of an $A$-module $V$ is a $B$-submodule $W$ admitting a $\mathbb K$-basis which is also an $\mathbb F$-basis of $V$. We shall also need the following concept.

\begin{defn}
Let $A$ and $B$ be as above.  A $B$-ample-form of a finite-dimensional $A$-module $V$ is a finite-dimensional $B$-submodule of $V$ which spans $V$ over $\mathbb F$.
\end{defn}

Let  $\gb\omega\in\cal P^+_{\mathbb F}$ and  let $v$ be a
highest-$\ell$-weight vector of a finite-dimensional
highest-$\ell$-weight $U(\tlie g)_\mathbb F$-module $V$ of highest
$\ell$-weight $\gb\omega$. Define
\begin{equation}
V^{f_\mathbb K}:=U(\tlie g)_{\mathbb K}v.
\end{equation}
$V^{f_\mathbb K}$ is clearly a $U(\tlie g)_\mathbb K$-submodule of $V$ and its isomorphism class does not depend on the choice of $v$. Let $\cal V:=U(\tlie h)_\mathbb Kv=\mathbb K(\gb\omega)v\cong\cal K(\gb\omega)$ as a $\mathbb K[\gb\Lambda]$-module. Then $V^{f_\mathbb K}$ is a highest-quasi-$\ell$-weight $U(\tlie g)_\mathbb K$-module  with highest-quasi-$\ell$-weight space $\cal V$ and highest-quasi-$\ell$ weight $[\gb\omega]$. Moreover, $V^{f_\mathbb K}$ is clearly a quotient of $W_\mathbb K(\gb\omega)$ and, in particular, it is a finite-dimensional $\mathbb K$-vector space. Thus, in order to show that $V^{f_\mathbb K}$ is a $U(\tlie g)_\mathbb K$-ample-form of $V$, it remains to show that it spans $V$ over $\mathbb F$.

\begin{prop}\label{p:forms}
Let  $\gb\omega\in\cal P^+_{\mathbb F}$ and $V$ a finite-dimensional highest-$\ell$-weight $U(\tlie g)_\mathbb F$-module of highest $\ell$-weight $\gb\omega$. Then $V^{f_{\mathbb K}}$ is a $U(\tlie g)_\mathbb K$-ample-form of $V$ and
\begin{equation}\label{e:forms0}
\dim_{\mathbb K}(V^{f_\mathbb K})\ge \deg(\gb\omega)\dim_\mathbb F(V).
\end{equation}
\end{prop}

\begin{proof}
Fix a basis $\{v^1,\dots,v^d\}$ of $\cal V=V^{f_\mathbb K}_{\wt(\gb\omega)}$. We will use the notation of the proof of Theorem \ref{t:wm}. In particular, it follows from that proof that there exist $\xi_j\in\Xi', j=1,\dots,\dim_\mathbb F(V)$, such that the set  $\{v^i_{\xi_j}:j=1,\dots,  \dim_\mathbb F(V)\}$ is an $\mathbb F$-basis of $V$ for each $i=1,\dots,d$. Evidently we have $v^i_{\xi_j}\in V^{f_\mathbb K}$ for all $i,j$. This shows that $V^{f_\mathbb K}$ spans $V$ over $\mathbb F$ and, together with the discussion preceding the proposition, it proves the first statement.
In particular, since $d=\deg(\gb\omega)$, to prove the second statement it suffices to show that the set $\{v^i_{\xi_j}:i=1,\dots,d, j=1,\dots,  \dim_\mathbb F(V)\}$ is linearly independent over $\mathbb K$.

Write $v^i=b_iv^1$ for some $b_i\in\mathbb F$ and observe that $\{b_i:i=1,\dots,m\}$ is linearly independent over $\mathbb K$. Let $a_{i,j}\in\mathbb K$ be such that $\sum_{i,j}a_{i,j}v^i_{\xi_j}=0$ and set $v_j=\sum_{i}a_{i,j}v^i_{\xi_j}$. Observe that $v_j=c_jv_{\xi_j}=b_jv_{\xi_j}^1$ where $c_j=\sum_{i}a_{i,j}b_i$. Since $\{v_{\xi_j}^{1}:j=1,\dots,\dim_\mathbb F(V)\}$ is an $\mathbb F$-basis of $V$, we have $c_j=0$ for all $j$. This in turn implies $a_{i,j}=0$ for all $i,j$.
\end{proof}

It is shown in Example \ref{ex:ample} below that the inequality in \eqref{e:forms0} can be strict even when $\gb\omega\in\cal P_\mathbb K^+$.

\begin{rem}
It follows from Theorem \ref{t:irpa}(c) that $V^{f_\mathbb K}$ is isomorphic to the restriction of the $U(\tlie g)_\mathbb L$-module $V^{f_{\mathbb L}}$ to $U(\tlie g)_\mathbb K$, where $\mathbb L=\mathbb K(\gb\omega)$.
\end{rem}

\begin{thm}\label{t:forms} Let  $\gb\omega\in\cal P^+_{\mathbb F}$.
\begin{enumerate}
\item $(V_\mathbb F(\gb\omega))^{f_\mathbb K}\cong V_\mathbb K(\gb\omega)$ and $\dim_\mathbb K(V_\mathbb K(\gb\omega))=\deg(\gb\omega)\dim_\mathbb F(V_\mathbb F(\gb\omega))$
\item $(W_\mathbb F(\gb\omega))^{f_\mathbb K}\cong W_\mathbb K(\gb\omega)$ and $\dim_\mathbb K(W_\mathbb K(\gb\omega))=\deg(\gb\omega)\dim_\mathbb F(W_\mathbb F(\gb\omega))$.
\end{enumerate}
\end{thm}

\begin{proof}
For both statements, let $\cal V$ be the highest-quasi-$\ell$-weight
space of the corresponding ample-form. If $V$ is irreducible, every
vector in $V^{f_\mathbb K}$ annihilated by $U(\tlie n^+)^0_\mathbb
K$ must be in $\cal V$ and, hence, the first statement of  part (a)
follows from the irreducibility of $\cal V$ as $U(\tlie h)_\mathbb
K$-module. Given the remark preceding the theorem, it suffices to
show the second statement in the case $\deg(\gb\omega)=1$. Since we
have an inclusion of $U(\tlie g)_\mathbb K$-modules $V_\mathbb
K(\gb\omega)\subseteq V_\mathbb F(\gb\omega)$, we also have an
inclusion of $U(\tlie g)_\mathbb F$-modules $V_\mathbb
K(\gb\omega)^\mathbb F\subseteq V_\mathbb
F(\gb\omega)\otimes_\mathbb K\mathbb F$. The assumption
$\deg(\gb\omega)=1$ implies the later is isomorphic, as $U(\tlie g)_\mathbb F$-module, to a direct sum
of $[\mathbb F:\mathbb K]$ copies of $V_\mathbb F(\gb\omega)$. In
particular, it follows that all highest $\ell$-weight vectors of
$V_\mathbb K(\gb\omega)^\mathbb F$ lie in its top weight space which
is one-dimensional. Hence, $V_\mathbb K(\gb\omega)^\mathbb F$ is an
irreducible $U(\tlie g)_\mathbb F$-module and the second statement
follows.

For (b) we already know that $(W_\mathbb F(\gb\omega))^{f_\mathbb K}$ is a quotient of $W_\mathbb K(\gb\omega)$ and that $\dim_{\mathbb K}((W_\mathbb F(\gb\omega))^{f_\mathbb K})\ge \deg(\gb\omega)\dim_\mathbb F(V)$.
In particular,
\begin{equation}\label{e:forms1}
\deg(\gb\omega)\dim_\mathbb F(W_\mathbb F(\gb\omega))\le \dim_\mathbb F(((W_\mathbb F(\gb\omega))^{f_\mathbb K})^\mathbb F)\le \dim_\mathbb F((W_\mathbb K(\gb\omega))^\mathbb F).
\end{equation}
We have similar equation for every $\gb\varpi\in[\gb\omega]$. Since $\deg(\gb\varpi)=\deg(\gb\omega)$ and $W_\mathbb K(\gb\varpi)\cong W_\mathbb K(\gb\omega)$ it follows that
\begin{equation}\label{e:forms2}
\deg(\gb\omega)\dim_\mathbb F(W_\mathbb F(\gb\varpi))\le \dim_\mathbb F(((W_\mathbb F(\gb\varpi))^{f_\mathbb K})^\mathbb F)\le \dim_\mathbb F((W_\mathbb K(\gb\omega))^\mathbb F).
\end{equation}
Fix $\gb\varpi\in[\gb\omega]_\mathbb K$ such that $\dim_\mathbb F(W_\mathbb F(\gb\varpi))\ge\dim_\mathbb F(W_\mathbb F(\gb\pi))$ for all $\pi\in[\gb\omega]_\mathbb K$

For the next argument, let $V$ be either $(W_\mathbb F(\gb\varpi))^{f_\mathbb K}$ or $W_\mathbb K(\gb\omega)$ and identify its highest-quasi-$\ell$-weight space with $\cal V$. Then,  $V^\mathbb F$ is generated by $\cal V^\mathbb F$ as a $U(\tlie g)_\mathbb F$-module. From Theorem \ref{t:irpa}(e) we know that $\cal V^\mathbb F$ has $\deg(\gb\omega)$ irreducible constituents each of which gives rise to an element $\gb\varpi\in [\gb\omega]$ when regarded as a $U(\tlie h)_\mathbb F$-module. Hence, $V^\mathbb F$ has a filtration
$0\subseteq W_1\subseteq W_2\subseteq\cdots\subseteq W_{\deg(\gb\omega)}=V^\mathbb F$ such that $W_j/W_{j-1}$ is a quotient of $W_\mathbb F(\gb\pi)$ for some $\gb\pi\in [\gb\omega]$.
It follows from the choice of $\gb\varpi$ that
\begin{equation}\label{e:forms3}
\dim_\mathbb F(V^\mathbb F)\le \deg(\gb\omega)\dim_\mathbb F(W_\mathbb F(\gb\varpi)).
\end{equation}
Using \eqref{e:forms3} with $V=(W_\mathbb F(\gb\varpi))^{f_\mathbb K}$ together with \eqref{e:forms2},
it follows that $\dim_\mathbb F(W_j)= \dim_\mathbb F(W(\gb\varpi))$ and, hence, $\dim_\mathbb F(W_\mathbb F(\gb\pi))=\dim_\mathbb F(W_\mathbb F(\gb\varpi))$ for all $\gb\pi\in[\gb\omega]$. It now follows from \eqref{e:forms1} together with \eqref{e:forms3} with $V=W_\mathbb K(\gb\omega)$ that $(W_\mathbb F(\gb\omega))^{f_\mathbb K}\cong W_\mathbb K(\gb\omega)$ and that all the above inequalities are equalities.
\end{proof}

\begin{rem}
The proof of part (b) of the above theorem would be simpler if we knew that $\dim_\mathbb F(W_\mathbb F(\gb\pi))=\dim_\mathbb F(W_\mathbb F(\gb\omega))$ for all $\gb\pi\in[\gb\omega]$. In fact it is expected that $\dim_\mathbb F(W_\mathbb F(\gb\omega))$ depends only on $\wt(\gb\omega)$. This is known to be true in characteristic zero and it is a consequence of the conjecture in \cite{jm} in the case of positive characteristic.
\end{rem}

Having in mind the arguments of the proof of Theorem \ref{t:forms}, part (a) of the following corollary is easily deduced from the remark after Corollary \ref{c:cig}. Part (b) follows from the proofs of Theorems \ref{t:irpa}(e) and \ref{t:forms}.

\begin{cor}
Let  $\gb\omega\in\cal P_\mathbb F^+$, $V=V_\mathbb F(\gb\omega)$, and $W=W_\mathbb F(\gb\omega)$.
\begin{enumerate}
\item  $\ch(V^{f_\mathbb K})=\deg(\gb\omega)\ch(V)$ and $\ch(W^{f_\mathbb K})=\deg(\gb\omega)\ch(W)$.
\item $(V^{f_\mathbb K})^\mathbb F\cong\opl_{\gb\omega'\in[\gb\omega]}^{} V^\mathbb F(\gb\omega')$ where $V^\mathbb F(\gb\omega')$ is an indecomposable self-extension of $V_\mathbb F(\gb\omega')$ of length $\deg(\gb\omega)/|[\gb\omega]|$. Similarly, $(W^{f_\mathbb K})^\mathbb F\cong\opl_{\gb\omega'\in[\gb\omega]}^{} W^\mathbb F(\gb\omega')$ where $W^\mathbb F(\gb\omega')$ is an indecomposable self-extension of $W_\mathbb F(\gb\omega')$ of length $\deg(\gb\omega)/|[\gb\omega]|$.
\text{}\hfill\qedsymbol
\end{enumerate}
\end{cor}

Let us stress that the reader should not be misled by these results. They do provide an effective tool for obtaining results for the finite-dimensional representation theory of $U(\tlie g)_\mathbb K$ from the knowledge of similar results for $U(\tlie g)_\mathbb F$ as we shall see in the next sections. However, it does not follow that the structures of $V_\mathbb K(\gb\omega)$ or $W_\mathbb K(\gb\omega)$ are identical (mutatis mutandis) to their counterparts over $\mathbb F$ as the example below shows in the case of Weyl modules. The case of irreducible modules is being treated in \cite{jmtpec} where we study their structure as $\mathbb K[\gb\Lambda]$-modules.

\begin{ex}
Let $\lie g=\lie{sl}_2,\mathbb K=\mathbb R, \iota$ be such that $\iota^2=-1$, and $\gb\omega(u)=(1-\iota u)^2$. The Weyl module $W_\mathbb C(\gb\omega)$ is a $4$-dimensional length-$2$ module with composition series given by $0\to \mathbb C\to W_\mathbb C(\gb\omega)\to V_\mathbb C(\gb\omega)\to 0$. By the theorem, $W_\mathbb R(\gb\omega)$ is an $8$-dimensional module having the $6$-dimensional module $V_\mathbb R(\gb\omega)$ as its irreducible quotient. Hence, a composition series for $W_\mathbb R(\gb\omega)$ is described by the exact sequence $0\to \mathbb R^{\oplus 2}\to W_\mathbb R(\gb\omega)\to V_\mathbb R(\gb\omega)\to 0$ showing that $W_\mathbb R(\gb\omega)$ has length 3. We shall see in Example \ref{ex:ample} that one can find $\gb\omega$ such that the length of $W_\mathbb K(\gb\omega)$  is smaller than that of $W_\mathbb F(\gb\omega)$. \end{ex}

We  can finally be precise about the last statement of Theorem \ref{t:ellhwrel}.

\begin{cor}
Let $V$ be a finite-dimensional highest-quasi-$\ell$-weight $U(\tlie g)_\mathbb K$-module with highest-quasi-$\ell$-weight $[\gb\omega]$ and let $\lambda=\wt(\gb\omega)$. Then
\begin{equation}\label{e:Lambdaonv}
\Lambda_{i,\lambda(h_i)}\Lambda_{i,-r}v = \Lambda_{i,\lambda(h_i)-r}v \quad\text{for all } v\in V_{\lambda}, i\in I, 0\le r\le \lambda(h_i).
\end{equation}
\end{cor}

\begin{proof}
This is a consequence of two facts. First, $V$ is a quotient of $(W_\mathbb F(\gb\omega))^{f_\mathbb K}$ for some $\gb\omega\in\cal P_\mathbb F^+$ by Theorem \ref{t:forms}. Second, by the last statement of \cite[Corollary 3.5(a)]{jm}, \eqref{e:Lambdaonv} is satisfied when $v$ is a highest-$\ell$-weight vector of $W_\mathbb F(\gb\omega)$.
\end{proof}

We now apply the theory of ample-forms to obtain a few results which will be used in Section \ref{ss:blocks}.

\begin{prop}\label{p:ellcharange}
If $V$ is a finite-dimensional $U(\tlie g)_\mathbb K$-module, its quasi-$\ell$-weights lie in $\cal P_{\mathbb F,\mathbb K}$.
\end{prop}

\begin{proof}
It suffices to prove the proposition for the irreducible modules $V_\mathbb K(\gb\omega)$, since the set of quasi-$\ell$-weights of $V$ is the union of the sets of quasi-$\ell$-weights of its irreducible constituents. On the other hand, Theorem \ref{t:forms} implies that it suffices to prove this in the case $\mathbb K=\mathbb F$, which is the subject of \cite[Corollary 3.6]{jm}.
\end{proof}

Since $U(\tlie g)_{\mathbb K}$ is a Hopf algebra, the dual vector space $V^*$ of a $U(\tlie g)_\mathbb K$-module $V$ can be given a structure of $U(\tlie g)_\mathbb K$-module where the action of
$x\in U(\tlie g)_\mathbb K$ on $f\in V^*$ is given by
$$(xf)(v) = f(S(x)v)$$
for all $v\in V$.

Recall that $w_0$ denotes the longest element of the Weyl group of $\lie g$. Given $\gb\varpi=\prod_j \gb\omega_{\mu_j,a_j}\in\cal P_\mathbb F, a_i\ne a_j$, set $\gb\varpi^* = \prod_j \gb\omega_{-w_0\mu_j,a_j}$.

\begin{prop}\label{p:dualaf}
$(V_\mathbb K(\gb\omega))^*\cong V_\mathbb K(\gb\omega^*)$ for all $\gb\omega\in\cal P^+_\mathbb F$.
\end{prop}

\begin{proof}
By Theorem \ref{t:forms} $V_\mathbb K(\gb\omega^*)\cong (V_\mathbb F(\gb\omega^*))^{f_\mathbb K}$ and $(V_\mathbb K(\gb\omega))^*\cong ((V_\mathbb F(\gb\omega))^*)^{f_\mathbb K}$, whereas
by \cite[Proposition 3.7]{jm} $(V_\mathbb F(\gb\omega))^*\cong V_\mathbb F(\gb\omega^*)$.
\end{proof}

\subsection{Tensor Products}\label{ss:tp}\hfill\\

We now prove some results on tensor products of finite-dimensional irreducible $U(\tlie g)_\mathbb K$-modules and give some examples illustrating cases not covered by these results. Tensor products will always be over the underlying field of the vector spaces involved.
Given two $U(\tlie g)_\mathbb K$-modules $V$ and $W$, we endow $V\otimes W$ with a $U(\tlie g)_\mathbb K$-action using the comultiplication as usual.

\begin{thm}\label{t:tp}
 Suppose $\gb\omega,\gb\varpi,\gb\pi\in\cal P_\mathbb F^+$ are such that $\gb\omega=\gb\varpi\gb\pi$ and $V_\mathbb F(\gb\omega)\cong V_\mathbb F(\gb\varpi)\otimes V_\mathbb F(\gb\pi)$. Then $V_\mathbb K(\gb\omega)\cong V_\mathbb K(\gb\varpi)\otimes V_\mathbb K(\gb\pi)$ iff $\deg(\gb\omega)=\deg(\gb\varpi)\deg(\gb\pi)$.
\end{thm}

\begin{proof}
Using Theorem \ref{t:forms}, we compute
\begin{align}\label{e:tpforms1}
&\dim_\mathbb K(V_\mathbb K(\gb\omega))=\deg(\gb\omega)\dim_\mathbb F(V_\mathbb F(\gb\omega))\\\label{e:tpforms2}
&\dim_\mathbb K(V_\mathbb K(\gb\varpi)\otimes_\mathbb K V_\mathbb K(\gb\pi)) = (\deg(\gb\varpi)\deg(\gb\pi))(\dim_\mathbb F(V_\mathbb F(\gb\omega))).
\end{align}
Thus, if $V_\mathbb K(\gb\omega)\cong V_\mathbb K(\gb\varpi)\otimes V_\mathbb K(\gb\pi)$, we must have $\deg(\gb\omega)=\deg(\gb\varpi)\deg(\gb\pi)$.

For the converse, since the hypothesis $V_\mathbb F(\gb\omega)\cong V_\mathbb F(\gb\varpi)\otimes V_\mathbb F(\gb\pi)$ and $\deg(\gb\omega)=\deg(\gb\varpi)\deg(\gb\pi)$ together imply that $\dim(V_\mathbb K(\gb\varpi)\otimes_\mathbb K V_\mathbb K(\gb\pi))=\dim(V_\mathbb K(\gb\omega))$, it suffices to show that $V_\mathbb K(\gb\varpi)\otimes V_\mathbb K(\gb\pi)$ contains a quotient of $W_\mathbb K(\gb\omega)$ as submodule.
Let $v_1$ and $v_2$ be highest-$\ell$-weight vectors of $V_\mathbb F(\gb\varpi)$ and $V_\mathbb F(\gb\pi)$, respectively. Then  $v:=v_1\otimes v_2$ is a highest-$\ell$-weight vector of $V_\mathbb F(\gb\omega)$. Also, $V_\mathbb K(\gb\varpi)\cong U(\tlie g)_\mathbb Kv_1$, $V_\mathbb K(\gb\pi)\cong U(\tlie g)_\mathbb Kv_2$, and $V_\mathbb K(\gb\omega)\cong U(\tlie g)_\mathbb Kv$. Let $\cal U=U(\tlie h)_\mathbb Fv=\mathbb Fv, \cal V=U(\tlie h)_\mathbb Kv=\mathbb K(\gb\omega)v$, $\cal V_1=U(\tlie h)_\mathbb Kv_1=\mathbb K(\gb\varpi)v_1$ and $\cal V_2=U(\tlie h)_\mathbb Kv_2=\mathbb K(\gb\pi)v_2$. Let $\varphi:\cal V_1\otimes_\mathbb K \cal V_2\to \cal U$ be the map of Lemma \ref{l:tpfields} after identifying $v,v_1,v_2$ with 1 in $\mathbb F, \mathbb K(\gb\varpi),\mathbb K(\gb\pi)$, respectively. By construction, $\varphi$ is a map of $U(\tlie h)_\mathbb K$-modules.
Since $\deg(\gb\omega)=\deg(\gb\varpi)\deg(\gb\pi)$, it follows from equation \eqref{e:chainfields} and Lemma \ref{l:tpfields} that $\varphi$ is injective. Hence we get inclusions $\cal V\subseteq \varphi(\cal V_1\otimes_\mathbb K\cal V_2)\subseteq \cal U$. Since $\varphi^{-1}(\cal V)$ is isomorphic to $\cal K(\gb\omega)$, it follows that $U(\tlie g)_\mathbb K\varphi^{-1}(\cal V)$ is a quotient of $W_\mathbb K(\gb\omega)$.
\end{proof}

\begin{defn}
If  $f(u)\in\mathbb K[u]$ is a non-constant polynomial with constant term $1$ and $\lambda\in P$, define the element $\gb\omega_{\lambda,f}\in\cal P_\mathbb K$ by $(\gb\omega_{\lambda,f})_i(u) = f(u)^{\lambda(h_i)}$.  When $f$ is of degree one, say $f(u)=1-au$, we may denote $\gb\omega_{\lambda,f}$ by $\gb\omega_{\lambda,a}$ instead.
Let $\cal I_\mathbb K$ be the set of all irreducible polynomials in $\mathbb K[u]$ with constant term $1$.
\end{defn}

Notice that  $\wt(\gb\omega_{\lambda,f})=\deg(f)\lambda$. Since $\mathbb K[u]$ is a unique factorization domain,  every $\gb\omega\in\cal P_\mathbb K^+$ can be written uniquely as
\begin{equation}\label{e:polyfactor}
\gb\omega = \prod_{j=1}^m \gb\omega_{\lambda_j,f_j}\quad\text{for some}\quad m\in\mathbb N, \lambda_j\in P^+, f_j\in\cal I_\mathbb K \text{ with } f_j\ne f_k\text{ if }j\ne k.
\end{equation}

Now let $\gb\omega\in\cal P_\mathbb F^+$ and observe that there exists a unique element $\gb\omega^\mathbb K\in\cal P_\mathbb K^+$ satisfying:
\begin{enumerate}
\item $\tilde{\gb\omega}:=\gb\omega(\gb\omega^\mathbb K)^{-1}\in\cal P_\mathbb F^+$,
\item $\gb\omega^\mathbb K$ is relatively prime to $\tilde{\gb\omega}$ in $\cal P_\mathbb F^+$,
\item $\gb\omega^\mathbb K(\gb\varpi)^{-1}\in\cal P_\mathbb K^+$ for every $\gb\varpi\in\cal P_\mathbb K^+$ satisfying properties (a) and (b).
\end{enumerate}
Notice that
\begin{equation}\label{e:degtilde}
\deg(\gb\omega^\mathbb K)=1 \quad\text{and}\quad \deg(\gb\omega)=\deg(\tilde{\gb\omega}).
\end{equation}
We will need the following notation. Given $\rho\in{\rm Aut}(\mathbb K)$ and $f\in\mathbb K[u], f(u)=\sum_r a_ru^r$, let $\rho(f)$ be the polynomial $\rho(f)(u)=\sum_r \rho(a_r)u^r$. Also, given a prime $p\in\mathbb N$,  let $P_p^+=\{\lambda\in P^+:\lambda(h_i)<p$ for all $i\in I\}$ and recall that $\tilde\Phi$ denotes the Frobenius homomorphism of Theorem \ref{t:frobenius}. Clearly, for every $\lambda\in P^+$ and $k\in\mathbb Z_+$, there exist unique $\lambda_k\in P_p^+$ (depending on $p$) such that
\begin{equation}\label{e:lambdap}
\lambda=\sum_{k\ge 0} p^k\lambda_k.
\end{equation}

\begin{cor}\label{c:tp}\hfill\\ \vspace{-10pt}
\begin{enumerate}
\item Suppose $\gb\varpi$ and $\gb\pi$ are relatively prime in $\cal P_\mathbb F^+$ and $\deg(\gb\varpi)\deg(\gb\pi)=\deg(\gb\varpi\gb\pi)$. Then $V_\mathbb K(\gb\varpi\gb\pi)\cong V_\mathbb K(\gb\varpi)\otimes V_\mathbb K(\gb\pi)$.
\item  Let $\gb\omega\in\cal P_\mathbb F^+$ and write $\gb\omega^\mathbb K = \prod_{j=1}^m \gb\omega_{\lambda_j,f_j}$ as in \eqref{e:polyfactor}. Then,
$$V_\mathbb K(\gb\omega)\cong V_\mathbb K(\gb\omega^\mathbb K)\otimes V_\mathbb K(\tilde{\gb\omega})\quad\text{and}\quad V_\mathbb K(\gb\omega^\mathbb K)\cong \otm_{j=1}^m V_\mathbb K(\gb\omega_{\lambda_j,f_j}).$$
\item If $\mathbb K$ has characteristic $p>0$, $f\in\cal I_\mathbb K$, $\lambda=\sum_k p^k\lambda_k\in P^+$ as in \eqref{e:lambdap}, and $\rho\in{\rm Aut}(\mathbb K)$ is the $p^{\text{th}}$-power map, then
$$V_\mathbb K(\gb\omega_{\lambda,f})\cong\otm_{k\ge 0}^{} V_\mathbb K(\gb\omega_{p^{k+l}\lambda_k,f}) \cong\otm_{k\ge 0}^{} (V_\mathbb K(\gb\omega_{\lambda_k,\rho^{k+l}(f)}))^{\tilde\Phi^{k+l}},$$
where $p^l=\deg(f)$ if $f$ is inseparable over $\mathbb K$ and  $l=0$ otherwise.
%\item Try to find a statement with $\tilde{\gb\omega}$ and Frobenius .............
\end{enumerate}
\end{cor}

\begin{proof}
Part (a) is immediate from Theorem \ref{t:tp} since $V_\mathbb F(\gb\varpi\gb\pi)\cong V_\mathbb F(\gb\varpi)\otimes V_\mathbb F(\gb\pi)$ in that case by \cite[Theorem 3.4 and Equation (3.3)]{jm}.  The same results from \cite{jm} also imply that both isomorphisms of part (b) hold when $\mathbb K=\mathbb F$ and part (b) follows from  Theorem \ref{t:tp} together with \eqref{e:degtilde}.

In part (c),  all the highest $\ell$-weights  involved are in $\cal P_\mathbb K^+$ and, hence, have degree one over $\mathbb K$. Moreover, taking pull-backs by $\tilde\Phi$ commutes with taking  $U(\tlie g)_\mathbb K$-ample-forms, i.e.,
$$V_\mathbb K(\gb\omega_{\lambda_k,a^{p^{k+l}}})^{\tilde\Phi^{k+l}}\cong \left(V_\mathbb F(\gb\omega_{\lambda_k,a^{p^{k+l}}})^{\tilde\Phi^{k+l}}\right)^{f_\mathbb K}.$$
Hence, part (c) follows from Theorem \ref{t:tp} if we show that the isomorphisms of (c) hold when $\mathbb K=\mathbb F$.
If $f$ is inseparable over $\mathbb K$ there exist $l\in\mathbb N$ and $a\in\mathbb F$ such that $f(u)=(1-au)^{p^l}=1-a^{p^l}u^{p^l}$. It follows that $\gb\omega_{\lambda,f}=\gb\omega_{p^l\lambda,a}$. On the other hand, if $f$ is separable, then  $f(u) = \prod_{j=1}^d(1-a_ju)$ for some distinct $a_j\in\mathbb F$, where $d=\deg(f)$. In both cases we are done by using \cite[Theorem 3.4]{jm} again.
\end{proof}

\begin{ex}\label{ex:tensor}
We now give an example illustrating what may happen when $\deg(\gb\varpi)\deg(\gb\pi)>\deg(\gb\varpi\gb\pi)$.
Let $\mathbb K=\mathbb R, \lie g=\lie{sl}_2$, $\gb\varpi(u)=1-\iota u$ and $\gb\pi(u)=1-2\iota u$, where $\iota^2=-1$. Then $V_\mathbb C(\gb\varpi\gb\pi)\cong V_\mathbb C(\gb\varpi)\otimes V_\mathbb C(\gb\pi)$.  Let $\gb\varpi'$ and $\gb\pi'$ be the other elements in $[\gb\varpi]$ and $[\gb\pi]$, respectively.
Then we have
$$V:= V_\mathbb R(\gb\varpi)\otimes V_\mathbb R(\gb\pi)\cong V_\mathbb R(\gb\varpi')\otimes V_\mathbb R(\gb\pi)\cong V_\mathbb R(\gb\varpi')\otimes V_\mathbb R(\gb\pi')\cong V_\mathbb R(\gb\varpi)\otimes V_\mathbb R(\gb\pi').$$
The above isomorphisms are manifestations of the composition series of $V^{\mathbb C}$:
$$V^\mathbb C\cong V_\mathbb C(\gb\varpi\gb\pi)\oplus V_\mathbb C(\gb\varpi'\gb\pi)\oplus V_\mathbb C(\gb\varpi'\gb\pi')\oplus V_\mathbb C(\gb\varpi\gb\pi').$$
Since $\gb\varpi\gb\pi\equiv \gb\varpi'\gb\pi'$ and $\gb\varpi'\gb\pi\equiv\gb\varpi\gb\pi'$, we conclude that
$$V\cong V_\mathbb R(\gb\varpi\gb\pi)\oplus V_\mathbb R(\gb\varpi'\gb\pi).$$
Now recall from \cite[Section 3.2]{jm} that $V_\mathbb C(\gb\varpi),V_\mathbb C(\gb\pi)$, etc., are evaluation representations and that it follows from \cite{cint,cpnew} (see also \cite[Corollary 3.5]{jm}) that every finite-dimensional irreducible $U(\tlie g)_\mathbb C$-module is isomorphic to a tensor product of evaluation representations with distinct spectral parameters. It would be quite natural to regard $V_\mathbb R(\gb\varpi)$ and $V_\mathbb R(\gb\pi)$ as analogues of evaluation representations even though they are not evaluation representations. The discussion about whether they have the ``same spectral parameter'' is more subtle. One option would be to say the spectral parameters are different since $2\iota$ does not belong to the conjugacy class of $\iota$ over $\mathbb R$. Another option would be to declare that they are the same since $\iota$ generates the same field extension of $\mathbb R$ as $2\iota$ does.
In any case, $V_\mathbb R(\gb\varpi\gb\pi)$ and $V_\mathbb R(\gb\varpi'\gb\pi)$ would be examples of irreducible $U(\tlie g)_\mathbb R$-modules which cannot be realized as a tensor product of ``evaluation representations''.
\end{ex}

\begin{ex}\label{ex:tensor2}
We give a second example which further illustrates how much richer the study of tensor products over non-algebraically closed fields is in comparison with the algebraically closed situation. Again, let $\mathbb K=\mathbb R, \lie g=\lie{sl}_2$, and $\iota^2=-1$, but this time set $\gb\varpi(u)=1-\iota u, \gb\pi=1+\iota u$, and $V=V_\mathbb R(\gb\varpi)\otimes V_\mathbb R(\gb\pi)$. Notice that $[\gb\varpi]=\{\gb\varpi,\gb\pi\}$ and, hence, $V=V_\mathbb R(\gb\varpi)^{\otimes 2}$. To shorten the notation, set $\gb\omega=\gb\varpi\gb\pi$. Recall (from \cite{jm} for instance) that
$$V_\mathbb C(\gb\varpi)\otimes V_\mathbb C(\gb\pi)\cong V_\mathbb C(\gb\omega), \quad V_\mathbb C(\gb\varpi)^{\otimes 2}\cong V_\mathbb C(\gb\varpi^2)\oplus \mathbb C, \quad V_\mathbb C(\gb\pi)^{\otimes 2}\cong V_\mathbb C(\gb\pi^2)\oplus \mathbb C.$$
It is not difficult to see that
$$V^\mathbb C\cong (V_\mathbb C(\gb\omega))^{\oplus 2}\oplus V_\mathbb C(\gb\varpi^2)\oplus V_\mathbb C(\gb\pi^2)\oplus \mathbb C^{\oplus 2}.$$
This in turn implies that
$$V\cong (V_\mathbb R(\gb\omega))^{\oplus 2}\oplus V_\mathbb R(\gb\varpi^2)\oplus \mathbb R^{\oplus 2}.$$
\end{ex}

We now use tensor products to give an example of strict inequality in \eqref{e:forms0}. As a biproduct we obtain an example showing that composition series of $W_\mathbb K(\gb\omega)$ may have strictly smaller lengths than those of $W_\mathbb F(\gb\omega)$.

\begin{ex}\label{ex:ample}
Let $\lie g=\lie{sl}_2, \mathbb K=\mathbb R$, and $\iota$  a square root of $-1$. Consider the non-split short exact sequence of $\tlie g$-modules
$$0\to \mathbb C\to M\to V_\mathbb C(\gb\omega_{\alpha,\iota})\to 0$$
as constructed in \cite[Proposition 3.4]{cmsc}. Here, $\alpha$ is the positive root of $\lie g$.
Let $V=M\otimes V_\mathbb C(\gb\omega_{\alpha,-\iota})$ so that
$$0\to V_\mathbb C(\gb\omega_{\alpha,-\iota})\to V\to V_\mathbb C(\gb\omega_{\alpha,\iota})\otimes V_\mathbb C(\gb\omega_{\alpha,-\iota})\to 0$$
is a short exact sequence. Thus, $V$ is a highest-$\ell$-weight module of length 2 and with highest $\ell$-weight $\gb\omega(u)=(1+u^2)^2\in\cal P^+_\mathbb R$. We claim that $\dim_\mathbb R(V^{f_\mathbb R})=15$. Since $\dim(V)=12$, it follows that $V^{f_\mathbb R}$ is an ample-form but not a form. The irreducible quotient of $V^{f_\mathbb R}$ is $V_\mathbb R(\gb\omega)$ whose dimension is $9=\dim_\mathbb C(V_\mathbb C(\gb\omega))$. On the other hand, the irreducible module $V_\mathbb R(\gb\omega_{\alpha,\iota})$ is isomorphic to the restriction of $V_\mathbb C(\gb\omega_{\alpha,\iota})$ to $U(\tlie g)_\mathbb F$, which is $6$-dimensional, and is an irreducible  constituent of $V^{f_\mathbb R}$. Thus, the dimension of $V^{f_\mathbb R}$ is at least $15$ which shows it is not a form. The completion of the proof that the dimension is indeed $15$ will also show that the length of $W_\mathbb R(\gb\omega)$ is 3, while the length of $W_\mathbb C(\gb\omega)$ is known to be 4.
In fact, the irreducible constituents of $W_\mathbb C(\gb\omega)$ are $V_\mathbb C(\gb\omega),V_\mathbb C(\gb\omega_{\alpha,\iota}), V_\mathbb C(\gb\omega_{\alpha,-\iota})$, and the trivial representation. By Theorem \ref{t:forms}(b), $W_\mathbb R(\gb\omega)\cong W_\mathbb C(\gb\omega)^{f_\mathbb R}$ is a $U(\tlie g)_\mathbb R$-form of $W_\mathbb C(\gb\omega)$ and, hence, its dimension is $16$. It follows that the irreducible constituents of $W_\mathbb R(\gb\omega)$ are $V_\mathbb R(\gb\omega),V_\mathbb R(\gb\omega_{\alpha,\iota})$, and the trivial representation (notice $\gb\omega_{\alpha,\iota}\in[\gb\omega_{\alpha,-\iota}]$).
\end{ex}

\begin{rem}
Since the finite-dimensional irreducible $U(\tlie g)_\mathbb F$-modules are products of evaluation representations, the study of tensor products of these modules is immediately  reduced to that of studying tensor products of finite-dimensional irreducible $U(\lie g)_\mathbb F$-modules. For the same reason, the study of $\ell$-characters of  finite-dimensional irreducible $U(\tlie g)_\mathbb F$-modules is reduced to that of characters of finite-dimensional irreducible $U(\lie g)_\mathbb F$-modules. The examples above show that these problems are not as trivial when $\mathbb K$ is not algebraically closed. In fact, the underlying combinatorics  has a rich relation with Galois theory. We will pursue a detailed study of these topics, as well as that of Jordan-H\"older series for $W_\mathbb K(\gb\omega)$, in the forthcoming publication \cite{jmtpec}.
\end{rem}

We end this subsection proving the analogue of Theorem \ref{t:tp} for Weyl modules.

\begin{thm}\label{t:wtp}
 Suppose $\gb\omega,\gb\varpi,\gb\pi\in\cal P_\mathbb F^+$ are such that $\gb\omega=\gb\varpi\gb\pi$ and $W_\mathbb F(\gb\omega)\cong W_\mathbb F(\gb\varpi)\otimes W_\mathbb F(\gb\pi)$. Then $W_\mathbb K(\gb\omega)\cong W_\mathbb K(\gb\varpi)\otimes W_\mathbb K(\gb\pi)$ iff $\deg(\gb\omega)=\deg(\gb\varpi)\deg(\gb\pi)$.
\end{thm}

\begin{proof}
The``only if'' part is proved similarly to that of Theorem \ref{t:tp}. Now let $v_1,v_2$ be highest-quasi-$\ell$-weight vectors of $W_\mathbb K(\gb\varpi)$ and $W_\mathbb K(\gb\pi)$, respectively, and set $v=v_1\otimes v_2\in W_\mathbb K(\gb\varpi)\otimes_\mathbb KW_\mathbb K(\gb\pi)$. Let also $v_1',v_2'$ be highest-$\ell$-weight vectors of $W_\mathbb F(\gb\varpi)$ and $W_\mathbb F(\gb\pi)$, respectively, and set $v'=v_1'\otimes v_2'\in W_\mathbb F(\gb\varpi)\otimes_\mathbb FW_\mathbb F(\gb\pi)\cong W_\mathbb F(\gb\omega)$. Supposing $\deg(\gb\omega)=\deg(\gb\varpi)\deg(\gb\pi)$ and using Lemma \ref{l:tpfields} similarly to the way we used it in the proof of Theorem \ref{t:tp} we conclude that $U(\tlie g)_\mathbb Kv$ is a quotient of $W_\mathbb K(\gb\omega)$. Since the hypothesis $\deg(\gb\omega)=\deg(\gb\varpi)\deg(\gb\pi)$ also implies that $\dim_\mathbb K(W_\mathbb K(\gb\omega)) = \dim_\mathbb K(W_\mathbb K(\gb\varpi)\otimes W_\mathbb K(\gb\pi))$, it suffices to show that $\dim_\mathbb K(U(\tlie g)_\mathbb Kv)\ge \dim_\mathbb K(W_\mathbb K(\gb\omega))$.

 Since $W_\mathbb F(\gb\omega)\cong W_\mathbb F(\gb\varpi)\otimes W_\mathbb F(\gb\pi)$, it follows that $U(\tlie g)_\mathbb Kv'\cong W_\mathbb K(\gb\omega)$. Hence, we are left to show that there exists a surjective linear map from $U(\tlie g)_\mathbb Kv$ to $U(\tlie g)_\mathbb Kv'$.  Consider the  $U(\tlie g)_\mathbb K$-module isomorphisms $\psi_1:W_\mathbb K(\gb\varpi)\to U(\tlie g)_\mathbb Kv_1'$ and $\psi_2:W_\mathbb K(\gb\pi)\to U(\tlie g)_\mathbb Kv_2'$ determined by sending $v_1$ to $v_1'$ and $v_2$ to $v_2'$, respectively. Then exists a unique $U(\tlie g)_\mathbb K$-homomorphism $\psi: W_\mathbb K(\gb\varpi)\otimes_\mathbb K W_\mathbb K(\gb\pi)\to W_\mathbb F(\gb\varpi)\otimes_\mathbb F W_\mathbb F(\gb\pi)$ such that $\psi(w_1\otimes w_2) = \psi_1(w_1)\otimes \psi_2(w_2)$. Moreover, $\psi(U(\tlie g)_\mathbb Kv)=U(\tlie g)_\mathbb K\psi(v)=U(\tlie g)_\mathbb Kv'$.
\end{proof}

\begin{rem}
Using the corollary of \cite[Conjecture 4.7]{jm} and Theorem \ref{t:wtp}, parts (a) and (b) of Corollary \ref{c:tp} above can be proved for Weyl modules in a similar way.
\end{rem}

\subsection{Blocks}\label{ss:blocks}\hfill\\

Let $\cal Q_\mathbb F$ be the subgroup of $\cal P_\mathbb F$ generated by the elements $\gb\omega_{\alpha,a}$ with $\alpha\in R^+$ and $a\in\mathbb F^\times$. Set $\Xi_\mathbb F=\cal P_\mathbb F/\cal Q_\mathbb F$ and let $\overline{\gb\varpi}$ denote the image of $\gb\varpi\in\cal P_\mathbb F$ in $\Xi_\mathbb F$.

\begin{lem}\hfill
\begin{enumerate}
\item The following conditions are equivalent for two elements $\gb\omega,\gb\varpi\in\cal P_\mathbb F$.
\begin{enumerate}
\item[(i)] There exists $\gb\eta\in\overline{\gb\varpi}$ such that $[\gb\eta]=[\gb\omega]$.
\item[(ii)] There exists $\gb\zeta\in[\gb\varpi]$ such that $\overline{\gb\zeta}=\overline{\gb\omega}$.
\end{enumerate}
\vspace{5pt}
\noindent If the above conditions hold we write $\gb\omega\sim\gb\varpi$.

\item The relation $\gb\omega\sim\gb\varpi$ is an equivalence relation on $\cal P_\mathbb F$.
\end{enumerate}
\end{lem}

\begin{proof}
We use  Lemma \ref{l:galoisonP} repeatedly without further mention.

Let $\gb\eta\in\cal P_\mathbb F, \gb\alpha\in\cal Q_\mathbb F$, and $g\in{\rm Aut}(\mathbb F/\mathbb K)$ be such that $\gb\eta = \gb\varpi\gb\alpha=g(\gb\omega)$. Then
$$\gb\omega = g^{-1}(g(\gb\omega)) = g^{-1}(\gb\varpi\gb\alpha) = (g^{-1}(\gb\varpi))(g^{-1}(\gb\alpha)).$$
Quite clearly $\gb\beta:=g^{-1}(\gb\alpha)\in\cal Q_\mathbb F$. Thus, setting  $\gb\zeta = g^{-1}(\gb\varpi)$, it follows that (i) implies (ii). The converse is proved similarly.

Reflexivity is clear. Assume $\gb\omega\sim\gb\varpi$. Then, using the characterization (i) of $\sim$, there exist $g\in{\rm Aut}(\mathbb F/\mathbb K)$ and $\gb\alpha\in\cal Q_\mathbb F$ such that $\gb\varpi\gb\alpha = g(\gb\omega)$. Hence, $\gb\omega\gb\beta=h(\gb\varpi)$ showing $\gb\varpi\sim\gb\omega$, where
$\gb\beta=(g^{-1}(\gb\alpha))^{-1}$ and $h=g^{-1}$. Now suppose $\gb\omega\sim\gb\varpi, \gb\varpi\sim\gb\pi$, and write $\gb\eta=\gb\varpi\gb\alpha=g(\gb\omega), \gb\zeta = \gb\pi\gb\beta = h(\gb\varpi)$ for some $\gb\alpha,\gb\beta\in\cal Q_\mathbb F, g,h\in{\rm Aut}(\mathbb F/\mathbb K)$. Then
$$\gb\omega = g^{-1}(g(\gb\omega)) = (g^{-1}(\gb\varpi))(g^{-1}(\gb\alpha)) = (g^{-1}(h^{-1}(\gb\pi\gb\beta)))(g^{-1}(\gb\alpha)) = ((hg)^{-1}(\gb\pi)) ((hg)^{-1}(\gb\beta))g^{-1}(\gb\alpha).$$
Setting $\gb\xi = (hg)^{-1}(\gb\pi)$ and $\gb\gamma= ((hg)^{-1}(\gb\beta))g^{-1}(\gb\alpha)\in\cal Q_\mathbb F$, the above becomes $\gb\xi=(hg)^{-1}(\gb\pi) = \gb\omega\gb\gamma^{-1}$ and transitivity follows.
\end{proof}

Let $\Xi_\mathbb K$ be the set of all equivalence classes of $\sim$ and, given $\gb\omega\in\cal P_\mathbb F$, let $\chi_\gb\omega$ be the corresponding element of $\Xi_\mathbb K$. Following the terminology of \cite{cmsc,jm}, we call $\Xi_\mathbb K$ the space of spectral characters of the category $\cal C_\mathbb K$ of finite-dimensional $U(\tlie g)_\mathbb K$-modules.

The next proposition depends on the conjecture made in \cite{jm}.

\begin{prop}\label{p:wdsc}
Let $V$ be an indecomposable finite-dimensional $U(\tlie g)_\mathbb K$-module. Then all the $\ell$-weights of $V^\mathbb F$ determine the same element in $\Xi_\mathbb K$.
\end{prop}

\begin{proof}
It follows from Theorem \ref{t:forms} and \cite[Corollary 4.12]{jm} (which depends on the conjecture) that the statement holds when $V$  is a Weyl module. The general case follows from arguments analogous to that of \cite[Section 5]{cmsc} making use of Proposition \ref{p:dualaf}.
\end{proof}

Given $\chi\in\Xi_\mathbb K$, let $\cal C_\chi$ be the abelian subcategory of $\cal C_\mathbb K$ consisting of those modules $V$ such that all the $\ell$-weights of $V^\mathbb F$ belong to $\chi$. It immediately follows from the previous proposition that
\begin{equation}
\cal C_\mathbb K =\opl_{\chi\in\Xi_\mathbb K}^{}\cal C_\chi.
\end{equation}

We now want to show that each of the categories $\cal C_\chi$ is indecomposable, i.e., that the above direct sum decomposition is the block decomposition of $\cal C_\mathbb K$. In order to do that we first review the argument used in \cite{cmsc,jm} to prove this in the case $\mathbb K=\mathbb F$. The first important step is the next proposition (see \cite[Proposition 3.4]{cmsc} in characteristic zero and \cite[Proposition 4.16]{jm} in positive characteristic).

We will use the following terminology. Given $\lambda,\mu\in P^+$, we say that $\lambda$ is directly linked to $\mu$ if ${\rm Hom}_{\lie g}(\lie g\otimes V_\mathbb C(\lambda),V_\mathbb C(\mu))\ne 0$.

\begin{prop}\label{p:basicext}
Suppose $\lambda$ and $\mu$ are directly linked, $\lambda>\mu$, and let $a\in\mathbb F^\times$. Then there exists a quotient of $W_\mathbb F(\gb\omega_{\lambda,a})$ having $V_\mathbb F(\gb\omega_{\mu,a})$ as a submodule.\hfill\qedsymbol
\end{prop}

We have the following corollary.

\begin{cor}[{\cite[Corollary 4.17]{jm}}]\label{c:basicext}
Let $a=a_0\in\mathbb F^\times, \lambda, \mu\in P^+$, and let $M$ be a quotient of $W_\mathbb F(\gb\omega_{\lambda,a})$ as in Proposition \ref{p:basicext}. For $j=1,\dots,k$, let $\nu_j\in P^+$ and $a_j\in \mathbb F^\times$ be such that $a_i\ne a_l$ for all $i,l = 0,\dots,k, i\ne l$. Then the $U(\tlie g)_\mathbb F$-module $M\otimes \left(\otimes_j V_\mathbb F(\gb\omega_{\nu_j,a_j})\right)$ is a quotient of $W_\mathbb F(\gb\omega)$ having $V_\mathbb F(\gb\varpi)$ as a submodule, where  $\gb\omega=\gb\omega_{\lambda,a}\prod_j \gb\omega_{\nu_j,a_j}$ and $\gb\varpi=\gb\omega_{\mu,a}\prod_j\gb\omega_{\nu_j,a_j}$.\hfill\qedsymbol
\end{cor}

Observe that  the elements $\gb\omega$ and $\gb\varpi$ of the above corollary define the same class $\chi$ in $\Xi_\mathbb F$. The corollary then says that the corresponding irreducible $U(\tlie g)_\mathbb F$-modules must lie in the same block of $\cal C_\mathbb F$. In order to prove that $\cal C_\chi$ is irreducible it suffices to obtain sequences of Weyl modules linking any two irreducible modules from $\cal C_\mathbb F$ (for a brief review on linkage on abelian categories see \cite[Section 1]{emec}). This is achieved by combining the above corollary with the next proposition  (cf. the proof of \cite[Proposition 2.3]{cmsc}).

\begin{prop}[{\cite[Proposition 1.2]{cmsc}}]\label{p:shaded}
Let $\mu,\lambda\in P^+$ be such that $\lambda-\mu\in Q$. Then,
there exists a sequence of weights $\mu_k \in P^+$, $k=0,\dots,
m$, with

\begin{enumerate}
\item $\mu_0 = \mu$,   $\mu_m=\lambda$,  and
\item $\mu_k$ is directly linked to $\mu_{k+1}$ for all $0\le k\le m$.\hfill\qedsymbol
\end{enumerate}
\end{prop}

Now let us return to studying the categories $\cal C_\chi$ with $\chi\in\Xi_\mathbb K$. Thus, let $\gb\omega,\gb\varpi\in\cal P_\mathbb F^+$ be such that $\gb\omega\sim\gb\varpi$ and write $\gb\varpi\gb\alpha = g(\gb\omega)$ for some $\gb\alpha\in\cal Q_\mathbb F$ and $g\in{\rm Aut}(\mathbb F/\mathbb K)$. We want to show that $V_\mathbb K(\gb\omega)$ is linked to $V_\mathbb K(\gb\varpi)$. We know that  $V_\mathbb K(\gb\omega)\cong V_\mathbb K(g(\gb\omega))$ and $V_\mathbb F(g(\gb\omega))$ is linked to $V_\mathbb F(\gb\varpi)$. In particular, there exists a sequence of Weyl modules $W_1,\dots, W_m$ linking $V_\mathbb F(g(\gb\omega))$ to $V_\mathbb F(\gb\varpi)$. Therefore, in order to conclude that $V_\mathbb K(\gb\omega)$ is linked to $V_\mathbb K(\gb\varpi)$, it suffices to prove the following proposition which is an analogue over $\mathbb K$ of Corollary \ref{c:basicext}.

\begin{prop}
Let $\gb\omega$ and $\gb\varpi$ be as in Corollary \ref{c:basicext}. Then there exists a quotient of the Weyl module $W_\mathbb K(\gb\omega)$ having $V_\mathbb K(\gb\varpi)$ as submodule.
\end{prop}

\begin{proof}
The argument is similar to the one in the proof of \cite[Proposition 4.16]{jm}. Namely, let $W$ be a quotient of $W_\mathbb F(\gb\omega)$ having $V_\mathbb F(\gb\varpi)$ as submodule and assume, without loss of generality, that $V_\mathbb F(\gb\varpi)$ is the unique simple submodule of $W$ of highest $\ell$-weight $\gb\varpi$.  Since $W^{f_\mathbb K}$ is a quotient of $W_\mathbb K(\gb\omega)$, it remains to show that there exists a nonzero $v'\in (W^{f_\mathbb K})_{\gb\varpi}$ such that $(x_{\alpha,r}^+)^{(k)}v'=0$ for all $\alpha\in R^+,r\in\mathbb Z$, and $k\in\mathbb N$.

From the proof of Proposition \ref{p:forms} we know that there exists a $\mathbb K$-basis of $W^{f_\mathbb K}$ formed by vectors of the form $(x_{\alpha_1,r_1}^-)^{(k_1)}\cdots (x_{\alpha_m,r_m}^-)^{(k_m)}v$ with $v\in (W^{f_\mathbb K})_\gb\omega$. Fix such a basis and let $v_1,\cdots, v_n$ be the elements of  this basis spanning $(W^{f_\mathbb K})_\gb\varpi$. Thus, any $v'$ as above must be a solution $v'=\sum_{j=1}^n c_jv_j$, for some $c_j\in\mathbb K$, of the linear system
$$(x_{\alpha,r}^+)^{(k)}\left(\sum_{j=1}^n c_jv_j\right)=0 \qquad\text{for all}\qquad \alpha\in R^+, r\in\mathbb Z, k\in\mathbb Z_+.$$
On one hand, by our hypothesis over $W$, this system has a
one-dimensional solution space over $\mathbb F$. On the other hand,
since $W^{f_\mathbb K}$ is invariant under $U(\tlie g)_\mathbb K$,
the linear equations obtained by writing each
$(x_{\alpha,r}^+)^{(k)}\left(\sum_{j=1}^n c_jv_j\right)$ in our
fixed basis have coefficients in $\mathbb K$. Hence, the system has
a nontrivial solution over $\mathbb K$.
\end{proof}

\begin{rem}
Let us remark a different feature regarding the relation between tensor products and blocks that we have when $\mathbb K$ is not algebraically closed. Namely, if $\chi_1,\chi_2\in\Xi_\mathbb F$ and $V_j\in\cal C_{\chi_j}, j=1,2$, then $V_1\otimes V_2\in\cal C_{\chi_1+\chi_2}$ (see \cite[Proposition 4.14]{jm}). When $\mathbb K\ne\mathbb F$, $\Xi_\mathbb K$ is not a group, but one could still expect that, given $\chi_1,\chi_2\in\Xi_\mathbb K$ and $V_j\in\cal C_{\chi_j}, j=1,2$,  $V_1\otimes V_2$ would belong to $\cal C_{\chi}$ for some $\chi\in\Xi_\mathbb K$. However, one easily checks that Examples \ref{ex:tensor} and \ref{ex:tensor2} supply counterexamples for this.
\end{rem}

\appendix
\setcounter{section}1
\setcounter{thm}0

\section*{Appendix: Proof of Theorem \ref{t:irpa}}

\subsection{Field Theory and Linear Algebra}

In this section we state several results on field theory and linear algebra to be used in the proof of Theorem \ref{t:irpa}. All of the results can be either found or easily deduced from classical algebra books such as \cite{hun}.

We start by recalling the following basic facts about roots of irreducible polynomials.

\begin{lem}\label{l:rootpoly}
Let $f,g$ be relatively prime irreducible polynomials in $\mathbb K[u]$.
\begin{enumerate}
\item $f$ and $g$ remain relatively prime as elements of $\mathbb F[u]$.
\item Either $f$ has $\deg(f)$ simple roots in $\mathbb F$ or $f=b(u-a)^{p^k}$ for some $b\in\mathbb K, a\in\mathbb F, k\in\mathbb N$, and $p>0$ is the characteristic of $\mathbb K$.
%\item ${\rm Aut}(\mathbb F/\mathbb K)$ permute transitively the roots of $f$.
\hfill\qedsymbol
\end{enumerate}
\end{lem}

\begin{lem}\label{l:insepgen}
Let $\mathbb K$ have characteristic $p>0$ and $a_1,\dots,a_m\in\mathbb F,m\in\mathbb N$, be elements whose irreducible polynomials over $\mathbb K$ are of the form $f_j(u)=(u-a_j)^{p^{k_j}}$ for some $k_j\in\mathbb Z_+$. There exists a unique maximal ideal $\cal M$ of $\mathbb K[x_1,\dots,x_m]$ containing the elements $f_j(x_j)$ and $\mathbb K[x_1,\dots,x_m]/\cal M\cong \mathbb K(a_1,\dots,a_m)$.\hfill\qedsymbol
\end{lem}

Notice that in the hypothesis of the above lemma either $a_j\in\mathbb K$ and $k_j=0$ or $a_j$ is purely inseparable over $\mathbb K$.

The next theorem is known as the Isomorphism Extension Theorem.

\begin{thm}\label{t:iet}
Let $\mathbb L$ be an algebraic extension of a field $\mathbb E$ and let $h:\mathbb E\to \mathbb E'$ be an isomorphism of fields. Then there exists an algebraic extension $\mathbb L'$ of $\mathbb E'$ and an isomorphism of fields $g:\mathbb L\to \mathbb L'$ such that $g(e)=h(e)$ for all $e\in\mathbb E$.\hfill\qedsymbol
\end{thm}

\begin{cor}\label{c:iet}
Let $\mathbb M\subseteq\mathbb E\subseteq\mathbb L$ be algebraic field extensions, $G={\rm Aut}(\mathbb L/\mathbb M)$, and  $H=\{g\in G: g(a)=a$ for all $a\in \mathbb E\}$. If $\mathbb E$ is $G$-stable, then $H$ is a normal subgroup of $G$ and ${\rm Aut}(\mathbb E/\mathbb M)\cong G/H$.\hfill\qedsymbol
\end{cor}

Recall that the hypothesis of the corollary is satisfied when $\mathbb E$ is a splitting field over $\mathbb M$.

Here are the properties of a (finite) Galois extension $\mathbb L/\mathbb K$ with Galois group $G={\rm Aut}(\mathbb L/\mathbb K)$ which will be relevant to us:
\begin{enumerate}
\item $|G| = [\mathbb L:\mathbb K]=\dim_\mathbb L(\mathbb K)$.
\item  $\mathbb L^G=\mathbb K$, where $\mathbb L^G$ is the fixed point set of the action of $G$ on $\mathbb L$.
\item There exists a normal basis of $\mathbb L$ over $\mathbb K$, i.e., there exists $a\in\mathbb L$ such that the orbit $G(a)$ is a basis of $\mathbb L$ over $\mathbb K$.
\item If $\mathbb M\subseteq \mathbb L$ is a field extension of $\mathbb K$, then $\mathbb L/\mathbb M$ is Galois.
\item $\mathbb L$ is a separable normal extension of $\mathbb K$.
\item $\mathbb L$ is the splitting field of a (finite) set of separable polynomials in $\mathbb K[x]$.
\end{enumerate}

Now we recall some linear algebra.

\begin{lem}\label{l:minpoly1}
Suppose $ V$ is a $\mathbb K[x]$-module and that $f\in\mathbb K[u]$ is a monic polynomial of minimal degree such that $f(x) V=0$. Then, $f$ divides every other $g\in\mathbb K[u]$ satisfying $g(x) V=0$.\hfill\qedsymbol
\end{lem}

A polynomial $f$ as in the lemma above is clearly unique when it exists. In that case, it is called the minimal polynomial of $x$ on $ V$.

\begin{prop}\label{p:linalg}
Let $f\in\mathbb K[u]$ be monic and nonconstant and $ V=\mathbb K[x]/(f(x))$.
\begin{enumerate}
\item $f$ is the minimal polynomial of $x$ on $ V$.
\item If $f=\prod_{j=1}^m f_j^{k_j}$ where $f_j$ are distinct monic irreducible polynomials and $k_j\in\mathbb N$, $ V\cong \opl_{j}^{} \frac{\mathbb K[x]}{(f_j(x)^{k_{j}})}.$
\item If $f=g^k$ for some irreducible polynomial $g$ and some $k\in\mathbb N$, then $ V$ is indecomposable and all its $k$ irreducible constituents are isomorphic to $\mathbb K[x]/(g(x))$.
\item If $\mathbb L$ is a field extension of $\mathbb K$, $ V^\mathbb L\cong \mathbb L[x]/(f(x))$.\hfill\qedsymbol
\end{enumerate}
\end{prop}

\begin{cor}\label{c:linalg}
Let $ V$ be a finite-dimensional $\mathbb K[x]$-module and $f\in\mathbb K[u]$ be the minimal polynomial of $x$ on $ V$. If $f$ is irreducible, $ V\cong (\mathbb K[x]/(f(x)))^{\oplus d}$, where $d=\dim( V)/\deg(f)$.\hfill\qedsymbol
\end{cor}

\subsection{Proof of Theorem \ref{t:irpa} (a)--(c)}\label{ss:aa-c}

Recall that $\cal K(\gb\varpi)=\mathbb K[X]v$ for some nonzero $v\in\cal F(\gb\varpi)$. In order to prove part (a) of the theorem it suffices to show that for any nonzero $w\in\cal K(\gb\varpi)$ there exists $g\in\mathbb K[X]$ such that $gw=v$. Now $w=fv=av$ for some $f\in\mathbb K[X]$ and some $a\in\mathbb F$. Let $h\in\mathbb K[u]$ be the irreducible polynomial of $a$ over $\mathbb K$. Since $a\ne 0$, $h(u) = b(b^{-1}u^n+b_{n-1}u^{n-1}+\cdots+b_1u+1)$ for some $b,b_j\in\mathbb K$ and some $n\in\mathbb N$. Setting $g=-b^{-1}f^{n-1}-b_{n-1}f^{n-2}-\cdots-b_1$, we have $gw = (-b^{-1}h(f)+1)v=v$ as desired.

To prove (b) and (c),  denote by $\varpi_f$ the eigenvalue of the action of $f\in\mathbb F[X]$  on $\cal F(\gb\varpi)$ and observe that the  map
$$\Phi_\gb\varpi:\cal K(\gb\varpi)\to\mathbb K(\gb\varpi), \qquad fv\mapsto \varpi_f$$
is an isomorphism of $\mathbb K$-vector spaces.
Let $\cal K(\gb\pi)=\mathbb K[X]w$ for some nonzero $w$ in $\cal F(\gb\pi)$ and suppose $\gb\pi=g(\gb\varpi)$ for some $g\in{\rm Aut}(\mathbb F/\mathbb K)$. Then $g$ induces a field isomorphism $g:\mathbb K(\gb\varpi)\to \mathbb K(\gb\pi)$ and, since $\gb\varpi,\gb\pi$ are multiplicative functionals, $\Phi_\gb\pi^{-1}\circ g\circ\Phi_{\gb\varpi}$ is an isomorphism of $\mathbb K[X]$-modules $\cal K(\gb\varpi)\to \cal K(\gb\pi)$. Conversely, if $\varphi:\cal K(\gb\varpi)\to \cal K(\gb\pi)$ is an isomorphism of  $\mathbb K[X]$-modules, then $\Phi_\gb\varpi^{-1}\circ \varphi\circ\Phi_{\gb\pi}$ is a field isomorphism $\mathbb K(\gb\varpi) \to \mathbb K(\gb\pi)$ fixing $\mathbb K$. We are done by Theorem \ref{t:iet}.

\subsection{Proof of Theorem \ref{t:irpa} (d) and (e)}

Recall that if  $\mathbb L/\mathbb E$ is a field extension and $ V$  an $\mathbb E$-vector space, then $ V^{\mathbb L}= V\otimes_{\mathbb E}\mathbb L$ and we identify $ V$ with $ V\otimes 1\subseteq  V^\mathbb L$. In this case, given $g\in G= {\rm Aut}(\mathbb L/\mathbb E)$, denote also by $g$ the $\mathbb E$-linear map $ V^{\mathbb L} \to  V^{\mathbb L}$ determined by $v\otimes a\mapsto v\otimes g(a)$. If  $\mathbb L/\mathbb E$ is a Galois extension we have $( V^\mathbb L)^G= V$. Notice that $g(av)=g(a)g(v)$ for all $g\in G, a\in\mathbb L$, and $v\in  V^\mathbb L$. In particular $g$ sends $\mathbb L$-subspaces of $ V^\mathbb L$ to $\mathbb L$-subspaces of $ V^\mathbb L$. If $A$ is an $\mathbb E$-algebra, $ V$   an $A$-module, and $\mathbb L/\mathbb E$  a field extension, then $ V^\mathbb L$ is an $A^\mathbb L$-module and the action of $G={\rm Aut}(\mathbb L/\mathbb E)$ on $ V^\mathbb L$ commutes with the action of $A$.
Conversely, if $ V$ is an $\mathbb L$-vector space we use the notation $ V^\mathbb E$ to mean the restriction of $ V$ to $\mathbb E$. It will always be clear from the context whether the notation $ V^\mathbb M$ means restriction of scalars or extension of scalars to $\mathbb M$.

\begin{lem}
If $ V$ is an irreducible finite-dimensional $\mathbb K[X]$-module, there exist finitely many $x_1,\dots,x_n$ $\in X$ such that $ V$ is irreducible as a $\mathbb K[x_1,\dots,x_n]$-module.
\end{lem}

\begin{proof}
Fix a nonzero $v\in V$. Since $V$ is irreducible, $ V=\mathbb K[X]v$. For each $x\in X$, let $ V_x=\{x^kv:0\le k <\deg(f_x)\}$, where $f_x$ is the minimal polynomial of $x$ on $V$. Then $\cup_{x\in X} V_x$ contains a $\mathbb K$-basis $B$ of $ V$ and, since $ V$ is finite-dimensional, there  exist $x_1,\dots,x_n\in X$ such that $B\subseteq \cup_{j=1}^n V_{x_j}$.
\end{proof}

From now on, we fix a  $d$-dimensional irreducible $\mathbb
K[X]$-module $ V$ with $d<\infty$ and a finite subset
$x_1,\dots,x_n$ of $X$ such that $ V$ is irreducible as a $\mathbb
K[x_1,\dots,x_n]$-module. We also fix the following notation. Let
$\cal M=\{f_j:j=1,\dots,n\}$ where $f_j$ is the minimal  polynomial
of $x_j$ on $ V$, let $\mathbb L$ be the splitting field of  $\cal
M$ inside $\mathbb F$, $\cal R_j$ the set of roots of $f_j$ in
$\mathbb L$, and $\mathbb E$ the separable closure of $\mathbb K$ in
$\mathbb L$. Notice that it follows from Proposition
\ref{p:linalg}(b) together with the commutativity of $\mathbb K[X]$
that $f_j$ is irreducible. We suppose that the ordering of the $x_j$
is such that  $f_j$ is separable iff $j>m$ for some $m\in\mathbb
Z_+$. Let also $G$ be the subgroup of ${\rm Aut}(\mathbb E/\mathbb
K)$ generated by all the elements with no fixed points in $\cal R_j$
for some $j>m$ and let $\mathbb M=\mathbb E^G$. In particular, $G$
is trivial and $\mathbb M=\mathbb E$ if $m=n$.

If $m>0$, then the characteristic of $\mathbb K$ is a prime number $p>0$. In that case, if $j\le m$, then $\cal R_j=\{a_j\}$ for some $a_j\in\mathbb L$, $f_j(u)=(u-a_j)^{p^{k_j}}$ for some $k_j\in\mathbb N$, and $\mathbb L=\mathbb E(a_1,\dots,a_m)$.

We first prove the following particular case of Theorem \ref{t:irpa}.

\begin{prop}\label{p:irpa-insep}
Suppose $\mathbb K=\mathbb E\ne \mathbb L$. Then $ U=\{v\in V^\mathbb L:x_jv=a_jv, \ \foral j=1,\dots,m\}$ is the unique irreducible $\mathbb L[x_1,\dots,x_m]$-submodule of $ V^{\mathbb L}$ and $ V\cong  U^\mathbb K$. In particular, $ V^\mathbb L$ is indecomposable and all its irreducible constituents are isomorphic to $ U$.
\end{prop}

\begin{proof}
By Corollary \ref{c:linalg}, $ V\cong (\mathbb
K[x_m]/(f_m(x_m)))^{\oplus d_m}$ as a $\mathbb K[x_m]$-module, where
$d_m=\dim( V)/\deg(f_m)$. Moreover, by Proposition
\ref{p:linalg}(d), $ V^{\mathbb K(a_m)}\cong (\mathbb
K(a_m)[x_m]/(f_m(x_m)))^{\oplus d_m}$ and, by Proposition
\ref{p:linalg}(c), $\mathbb K(a_m)[x_m]/(f_m(x_m))$ is an
indecomposable $\mathbb K(a_m)[x_m]$-module and its irreducible
constituents are all isomorphic to $\mathbb K(a_m)[x_m]/(x_m-a_m)$.
In particular, all $\mathbb K[X]$-submodules of $ V^{\mathbb
K(a_m)}$ lie in the set $ U_m=\{v\in V^{\mathbb K(a_m)}:x_mv=a_mv\}$
which is a $d_m$-dimensional $\mathbb K(a_m)$-vector space.

Now Lemma \ref{l:insepgen} implies that $\dim_\mathbb K(V)=[\mathbb
K(a_1,\dots,a_m):\mathbb K]$. Since $[\mathbb K(a_m):\mathbb
K]=\deg(f_m)$ and
$$[\mathbb K(a_1,\dots,a_m):\mathbb K] = [\mathbb K(a_m)(a_1,\dots,a_{m-1}):\mathbb K(a_m)][\mathbb K(a_m):\mathbb K],$$
 we see that $d_m=[\mathbb K(a_m)(a_1,\dots,a_{m-1}):\mathbb K(a_m)]$. On the other hand, parts (a) and (d)  of Proposition \ref{p:linalg} imply that $f_j$ is the minimal polynomial of $x_j$ on $U_m$. Hence, $U_m$ has an irreducible $\mathbb K(a_m)[x_1,\dots,x_{m-1}]$-submodule isomorphic to the quotient of $\mathbb K(a_m)[x_1,\dots,x_{m-1}]$ by the unique maximal ideal containing the elements $g_j(x_j)$, where $g_j(u)$ is the irreducible polynomial of $a_j$ over $\mathbb K(a_m)$. Using Lemma \ref{l:insepgen}, we see that the $\mathbb K(a_m)$-dimension of this irreducible $\mathbb K(a_m)[x_1,\dots,x_{m-1}]$-module is $[\mathbb K(a_m)(a_1,\dots,a_{m-1}):\mathbb K(a_m)]=d_m$. In particular, it follows that $U_m$ is the unique irreducible $\mathbb K[X]$-submodule of $V^{\mathbb K(a_m)}$ and $U_m^\mathbb K\cong V$.

The proposition now follows from an easy induction on $m$.
\end{proof}

We now prove:

\begin{prop}\label{p:fdirpa} Let $ U=\mathbb Lu$ be an irreducible $\mathbb L[x_1,\dots,x_n]$-submodule of $ V^\mathbb L$ and $a_j$ be the eigenvalue of $x_j$ on $u, j=1,\dots,n$.
\begin{enumerate}
\item ${\rm Aut}(\mathbb E/\mathbb M)={\rm Aut}(\mathbb L/\mathbb M)=G$, $\mathbb E=\mathbb M(\bigcup\limits_{j>m}\cal M_j)$, $\mathbb L=\mathbb E(a_1,\dots,a_m)$, and $ V^\mathbb M$ is an irreducible $\mathbb M[x_1,\dots,x_n]$-module.
\item $d=[\mathbb L:\mathbb M]$, $ U^\mathbb E$ is an irreducible $\mathbb E[x_1,\dots,x_n]$-module, and $ V^\mathbb E\cong\opl_{g\in G}^{} g( U^\mathbb E)$. In particular, $ V^\mathbb L\cong\opl_{g\in G}^{} (g( U^\mathbb E))^{\mathbb L}$  and $(g( U^\mathbb E))^{\mathbb L}$ is an indecomposable $\mathbb L[x_1,\dots,x_n]$-module all of whose $d/|G|$ irreducible constituents are isomorphic to $g( U^\mathbb E)$.
\item $ U^\mathbb M$ is isomorphic to $ V^\mathbb M$ as an $\mathbb M[x_1,\dots,x_n]$-module.
\item The $\mathbb K[x_1,\dots,x_n]$-submodule of $ U$ generated by $u$ is isomorphic to $ V$.
\end{enumerate}
\end{prop}

\begin{proof}
Since $\mathbb L$ is a splitting field of a finite set of polynomials and $\mathbb E$ is the separable closure of $\mathbb K$ in $\mathbb L$, it follows that  $\mathbb L=\mathbb E(a_1,\dots,a_m)$, ${\rm Aut}(\mathbb L/\mathbb E)$ is trivial, and $\mathbb E$ is the splitting field of $\{f_j:j>m\}$ inside $\mathbb L$. In particular, $\mathbb E/\mathbb K$ is a Galois extension. By construction of $\mathbb M$, $\cal M_j\cap\mathbb M=\emptyset$ for all $j=1,\dots,n$, $f_j$ remain irreducible as elements of $\mathbb M[x]$ for all $j$, and $\mathbb E/\mathbb M$ is a Galois extension with Galois group $G$. The remaining statements of part (a) follow.

As seen in the proof of Proposition \ref{p:irpa-insep}, $ U^\mathbb E$ is a $[\mathbb L:\mathbb E]$-dimensional irreducible $\mathbb E[x_1,\dots,x_m]$-module isomorphic to the quotient of $\mathbb E[x_1,\dots,x_m]$ by the unique maximal ideal containing $f_1(x_1),\dots, f_m(x_m)$. Since $a_j\in\mathbb E$ for $j>m$, it follows from Lemma \ref{l:insepgen} that $ U^\mathbb E$ is isomorphic to the quotient of $\mathbb E[x_1,\dots,x_n]$ by the maximal ideal containing $f_j(x_j), j\le m, x_j-a_j, j>m$. Let $ U'=\{v\in V^\mathbb E: (x_j-a_j)v=0, j>m, f_j(x_j)v=0, j\le m\}$. It follows that there is an inclusion $ U^\mathbb E$ into $ U'$ and, thus, we identify $ U^\mathbb E$ with an irreducible submodule of $ U'$.

Quite clearly, for every $g\in G$,  the map $g: U^\mathbb E\to g( U^\mathbb E)$ is an isomorphism of $\mathbb M[x_1,\dots,x_n]$-modules. Since $\mathbb E/\mathbb M$ is Galois, we have $\mathbb E=\mathbb M[a_{m+1},\dots,a_n]$. Fix a normal basis $G(a)$ of $\mathbb E$ over $\mathbb M$ and let $f\in\mathbb M[x_{m+1},\dots,x_n]$ satisfy $f[a_{m+1},\dots,a_n]=a$. Then, for each nonzero $w\in U^\mathbb E$, $f[x_{m+1},\dots,x_n]g(w) = g(f[x_{m+1},\dots,x_x]w) = g(a)g(w)$ showing that the vectors $g(w), g\in G$, have distinct eigenvalues for the action of $f[x_{m+1},\dots,x_n]$ and, therefore, are linearly independent over $\mathbb E$. Let
$$w' = \sum_{g\in G} g(w)\qquad\text{and}\qquad  V' = \opl_{g\in G}^{}g( U^\mathbb E)\subseteq  V^\mathbb E.$$
Thus $w'\ne 0$ and $g(w)=w$. Moreover,  $w'\in  V'\cap  V^\mathbb M$ since $ V^\mathbb M=( V^\mathbb E)^G$. Now the irreducibility of $ V^\mathbb M$ implies
$$ V'\cap  V^\mathbb M =  V^\mathbb M \qquad\text{and}\qquad  V'= V^\mathbb E.$$
It follows that $[\mathbb L:\mathbb M]=[\mathbb L:\mathbb E]|G|=\dim_{\mathbb E}( V')=\dim_{\mathbb E}( V^\mathbb E) = d$ and $ U'\cong  U^\mathbb E$. The remainder of part (b) now follows from Proposition \ref{p:irpa-insep}.

Let $\varphi:{ U}^\mathbb M\to  V^\mathbb M$ be the $\mathbb M$-linear map sending each $w\in  U^\mathbb E$ to $w'=\sum_{g\in G} g(w)$. One immediately sees that $\varphi$ is a homomorphism of $\mathbb M[x_1,\dots,x_n]$-modules. Since $\varphi$ is clearly injective and $\dim_{\mathbb M}({ U}^\mathbb M)=[\mathbb L:\mathbb M]=d$, we are done with part (c).

Finally, $ V^\mathbb M$ is isomorphic to the direct sum of $[\mathbb M:\mathbb K]$ copies of $ V$ as a $\mathbb K[x_1,\dots,x_n]$-module. Thus, any nonzero vector $v\in  V^\mathbb M$ generates a $\mathbb K[x_1,\dots,x_n]$-submodule isomorphic to $ V$ and (d) follows from (c).
\end{proof}

To complete the proof of Theorem \ref{t:irpa}, observe that every $x\in X$ must preserve the one-dimensional spaces $\mathbb Lg(u)$ for all $g\in G$. Hence, for every $x\in X$, $xu=\varpi_xu$ for some $\varpi_x\in\mathbb L$. Since $\dim_\mathbb K(V)=[\mathbb L:\mathbb M]$, it follows that $\varpi_x\in\mathbb K(a_1,\dots,a_n)$ for all $x\in X$. In particular, $V\cong \cal K(\gb\varpi)$ with $\gb\varpi=(\varpi_x)_{x\in X}$. It remains to show that $[\gb\varpi]=\{g(\gb\varpi):g\in G\}$. But this is clear since the extensions $\mathbb K\subseteq\mathbb L\subseteq\mathbb F$ satisfy the condition of Corollary  \ref{c:iet}. This completes the proof of part (d) and part (e) is now immediate.\ \hfill\qedsymbol

\bibliographystyle{amsplain}

\end{document}